\documentclass{amsart}

\usepackage{graphicx}
\usepackage{algorithm}
\usepackage{algorithmic}
\usepackage{url}

\newtheorem{theorem}{Theorem}[section]
\newtheorem{corollary}[theorem]{Corollary}

\newtheorem{lemma}[theorem]{Lemma}

\theoremstyle{definition}
\newtheorem{condition}{Condition}

\theoremstyle{remark}
\newtheorem{remark}{Remark}

\begin{document}

\title[Exceptional surgeries on alternating knots]{Exceptional surgeries on alternating knots}



\author{Kazuhiro Ichihara}
\address{Department of Mathematics, 
College of Humanities and Sciences, Nihon University,
3-25-40 Sakurajosui, Setagaya-ku, Tokyo 156-8550, Japan}
\email{ichihara@math.chs.nihon-u.ac.jp}
\thanks{Ichihara is partially supported by JSPS KAKENHI Grant Number 23740061.}

\author{Hidetoshi Masai}
\address{Graduate School of Mathematical Sciences,
The University of Tokyo,
3-8-1 Komaba Meguro-ku,
Tokyo 153-8914,
Japan}
\email{masai@ms.u-tokyo.ac.jp}
\thanks{The work of Masai was partially supported by JSPS Research Fellowship for Young Scientists.}

\dedicatory{Dedicated to Professor Sadayoshi Kojima on the occasion of his 60th birthday}

\begin{abstract}
We give a complete classification of 
exceptional surgeries on hyperbolic alternating knots in the 3-sphere. 
As an appendix, we also show that 
the Montesinos knots $M (-1/2, 2/5, 1/(2q + 1))$ with $q \ge 5$ 
have no non-trivial exceptional surgeries. 
This gives the final step in 
a complete classification of exceptional surgeries on arborescent knots.
\end{abstract}

\maketitle

\setcounter{tocdepth}{1}
\tableofcontents

\section{Introduction}

The well-known Hyperbolic Dehn Surgery Theorem 
due to Thurston \cite[Theorem 5.8.2.]{Thur} says that 
each hyperbolic knot (i.e., a knot with the complement admitting a hyperbolic structure) 
admits only finitely many Dehn surgeries yielding non-hyperbolic manifolds. 
In view of this, such finitely many exceptions are called \textit{exceptional surgeries}. 


A complete classification of exceptional surgeries on hyperbolic knots in the 3-sphere remains an important and difficult challenge in both Knot Theory and 3-manifold topology. However, such a classification is known for some infinite families of knots. For example, a classification of exceptional surgeries on hyperbolic 2-bridge knots was obtained in 
\cite{BrittenhamWu2001}. Quite recently, exceptional surgeries on hyperbolic pretzel knots were also classified in \cite{Meier}. 

In this paper, we consider exceptional surgeries on 
hyperbolic alternating knots in $S^3$, 
one of the most well-known classes of knots. 
A knot in $S^3$ is called \textit{alternating} if 
it admits a diagram with alternatively arranged 
over-crossings and under-crossings running along it.

The main theorem of this paper is:

\begin{theorem}\label{MainThm}
Let $K$ be a hyperbolic alternating knot in the $3$-sphere. 
If $K$ admits a non-trivial exceptional surgery, 
then $K$ is equivalent to an arborescent knot. 
\end{theorem}

The definition of arborescent knots together with other definitions and background is 
delayed until $\S$\ref{sect:Preliminaries}.

\bigskip

Recently, exceptional surgeries on 
(both alternating and non-alternating) arborescent knots have been almost classified. 
Building on these partial results,
we provide a complete classification 
of exceptional surgeries on alternating knots as a corollary of our theorem.  

\begin{corollary}\label{Cor}
Let $K$ be a hyperbolic alternating knot in $S^3$. 
Suppose that the manifold $K(r)$ obtained by Dehn surgery on $K$ along a non-trivial slope $r$ 
is non-hyperbolic for some rational number $r$. 
Then $r$ must be an integer and $K(r)$ is irreducible. 
Furthermore the following hold. 
If $K(r)$ is toroidal, then $K(r)$ is not a Seifert fibered, and $K$ is equivalent to either 
\begin{itemize}
\item
the figure-eight knot and $r=0, \pm 4$, 
\item
a two bridge knot $K_{[b_1,b_2]}$ with $|b_1|,|b_2| > 2$, 
and $r=0$ if both $b_1 , b_2$ are even, 
$r = 2 b_2$ if $b_1$ is odd and $b_2$ is even, 
\item
a twist knot $K_{[2n,\pm 2]}$ with $|n| > 1$ and $r= 0, \mp 4$,
\item
a pretzel knot $P(q_1,q_2,q_3)$ with $q_i \ne 0 , \pm 1$ for $i=1,2,3$, 
and $r=0$ if $q_1,q_2,q_3$ are all odd, 
$r=2(q_2+q_3)$ if $q_1$ is even and $q_2,q_3$ are odd. 
\end{itemize}
In the above, when $r\ne0$, then $r$ is always a boundary slope of a once punctured Klein bottle spanned by $K$. 
If $K(r)$ is small Seifert fibered, then $K(r)$ has the infinite fundamental group, and 
$K$ is equivalent to either 
\begin{itemize}
\item
the figure-eight knot and $r= \pm 1 , \pm 2 , \pm 3$, 
\item
a twist knot $K_{[2n,\pm 2]}$ with $|n| > 1$ and $r= \mp 1, \mp 2, \mp 3$.
\end{itemize}
In particular, the figure-eight knot is the only knot 
admitting 10 exceptional surgeries among hyperbolic alternating knots, 
and the others admit at most 5 exceptional surgeries. 
\end{corollary}

The proof of this corollary is given in Section \ref{sec.3}. 
The last assertion immediately follows from the classification above, 
which gives an affirmative solution for alternating knots 
to the famous Gordon's conjecture: 
A hyperbolic manifold admits 10 exceptional fillings if and only if it is the figure 8 knot complement, and otherwise, it admits fewer exceptional fillings. 
It was shown in \cite{Ichihara2008AGT} based on \cite{Ichihara2008JKTR} that 
any hyperbolic alternating knot in $S^3$ has at most 10 exceptional surgeries, 
and recently, Lackenby and Meyerhoff \cite{LackenbyMeyerhoff2013} proved in general that 
any hyperbolic knot in any closed 3-manifold has at most 10 exceptional surgeries. 
On the other hand, as stated in \cite[Problem~1.77]{KirbyList}, 
Gordon conjectured that 
only the figure-eight knot attains the maximal, that is 10, 
but the methods used in the above papers could not prove this. 

\bigskip

We remark that our proof of Theorem \ref{MainThm} is computer-aided. 
In Section~\ref{sec.4}, we will give an outline of our proof, and, 
in particular, we will clarify where we used computer in the proof. 
Actually, due to \cite{lack}, 
we have only finitely many (but a huge number of) links 
so that Theorem \ref{MainThm} follows from 
the complete classification of certain types of exceptional surgeries on them. 
Thus our task is to investigate the surgeries on these finite number of links. 

In Section \ref{sec.symmetry}, 
we will discuss how to reduce the number of the links which we have to check. 
As is explained in Section 6, by computer-aided calculations, 
we have a potential procedure to rigorously verify that 
a given link admits no non-trivial exceptional surgeries. 
However applying it for all the links obtained by using \cite{lack} 
alone 
is computationally expensive.
Therefore we give a number of observations
to reduce the number of links we need to check, 
which we estimated to be in the millions.
First, we consider symmetries of the links and the other diagrammatic arguments 
in order to reduce this number. 
To further reduce the number of links and the number of components for some of the links, 
we applied some techniques and results of \cite{Wu2012}, which use essential laminations 
in the link exteriors. 
Even with these reductions, the size of the computation is outside the scope of a personal
computer. As noted above, the verification needed for each link is an involved process.
To be more specific, 
we have about 30,000 links to investigate, and
 for each link, 
we have to apply the procedure developed in \cite{HIKMOT} recursively. 
In fact, in the worst case, 
we have to apply the procedure more than 18,000 times. 
Therefore, we ran our computations on the super-computer, ``TSUBAME'', 
housed at Tokyo Institute of Technology. 
The result of all the computations verifies that none of the links 
admit unexpected non-trivial exceptional surgeries. 

In Section \ref{sec.codes}, we will explain our main code \texttt{fef.py} 
(short for \underline{f}ind \underline{e}xceptional \underline{f}illings) that rigorously
ensures the non-existence of non-trivial exceptional surgeries of certain type.
The code and the outputs of the program are downloadable from \cite{Webpage}.
We will explain the procedure of our code in detail, 
including information about computation environments used during our computation. 
Our program is essentially based on the technique developed in
\cite{MartelliPetronioRoukema}\footnote{Recently, they updated their paper and code to use our 
technique. See Version 2 of \cite{MartelliPetronioRoukema}.}.
The code for the first version of \cite{MartelliPetronioRoukema} did not account for round-off error properly.
To obtain mathematically rigorous computations, 
we improved their code using verified numerical analysis 
based on interval arithmetic. 
Some fundamentals about such methods will be given in Appendix \ref{ape.A}.
The key step to show that the links have no exceptional surgeries
is to prove the hyperbolicity of a given manifold rigorously. 
A technique to prove the hyperbolicity via computer
has been developed in \cite{HIKMOT} 
by the team containing the authors of this paper. 

Actually our computer assisted part of the proof of Theorem \ref{MainThm} can be 
adapted to provide the final step 
needed to classify all exceptional surgeries on hyperbolic arborescent knots. 
In Appendix \ref{ape.B}, applying our method of obtaining 
a complete classification of exceptional surgeries on a given hyperbolic link, 
we show that the Montesinos knots 
$M (-1/2, 2/5, 1/(2q + 1))$ with $q \ge 5$ 
have no non-trivial exceptional surgeries. 
This gives the final step in 
a complete classification of exceptional surgeries on arborescent knots.
We remark that these computations can be done without using the super-computer ``TSUBAME''.

\begin{remark}
We here remark that 
prime alternating knots are known to be all hyperbolic 
except for $(2,p)$-torus knots, 
that is, knots isotoped to the $(2,p)$-curves 
on the standardly embedded torus in $S^3$. 
Actually Menasco showed 
in \cite[Corollary 2]{Menasco1984} 
a non-split prime alternating link which is not a torus link has 
the complement admitting a complete hyperbolic structure of finite volume, 
and Murasugi showed in \cite[Theorem 3.2]{Murasugi1958} that 
the torus knots of type $(2,p)$ are only alternating knots among all torus knots 
by calculating Alexander polynomials. 
Also note that a purely geometric proof of the latter was obtained by 
Menasco and Thistlethwaite \cite[Corollary 2]{MenascoThistlethwaite1992}. 
\end{remark}

\begin{remark}
For the proof of Theorem \ref{MainThm}, we used the code named\linebreak ``fef.py" which is specially customized for alternating knots.
The code named ``fef\_gen.py" which we used in Appendix \ref{ape.B} works for any cusped hyperbolic 3-manifolds.
Both code are included the package available at \cite{Webpage}.
\end{remark}


\section{Preliminaries}\label{sect:Preliminaries}

\subsection{Dehn surgery}

By a \textit{Dehn surgery} on a knot $K$, 
we mean the following operation to create a new $3$-manifold 
from a given one and a given knot: 
first remove an open tubular neighborhood of $K$ to obtain 
the \textit{exterior} $E(K)$ of $K$, and 
glue a solid torus $V$ back via a boundary homeomorphism 
$f : \partial V \rightarrow \partial E(K)$. 
We say the isotopy class of each non-trivial unoriented simple closed curve in 
$\partial E(K)$ is a \emph{slope}. 
We pay special attention to the slope
$\gamma$ that is identified to the isotopy class of curves 
in $\partial V$ that bounds a disk in $V$. 
In this context, we call $\gamma$ the \textit{surgery slope} 
(see \cite{RolfsenBook} for further details
and background on Dehn surgery). 
When $K$ is a knot in $S^3$, 
by using the standard meridian-longitude system, 
slopes on $\partial E(K)$ are parametrized by 
$\mathbb{Q} \cup \{1/0\}$. 
For example, the meridian of $K$ corresponds to $1/0$ and the longitude to $0$. 
We thus denote by $K(r)$ the 3-manifold obtained by 
Dehn surgery on a knot $K$ along a slope corresponding to a rational number $r$. 
By the \textit{trivial} Dehn surgery on $K$ in $S^3$, 
we mean the Dehn surgery on $K$ along the meridional slope $1/0$. 
Thus, it yields $S^3$ again, which is obviously exceptional, when $K$ is hyperbolic. 
We say that a Dehn surgery on $K$ in $S^3$ is \textit{integral} 
if it is along an integral slope. 
This means that the curve representing the surgery slope 
runs longitudinally once.

We also recall a classification of exceptional surgeries. 
As a consequence of 
the famous Geometrization Conjecture, 
raised by Thurston in \cite[section 6, question 1]{Thurston1982}, 
and established by Perelman's works, \cite{Perelman1}, \cite{Perelman2}, \cite{Perelman3}, 
all closed orientable $3$-manifolds are classified as: 
reducible (i.e., containing 2-spheres not bounding 3-balls), 
toroidal (i.e., containing incompressible tori), 
Seifert fibered (i.e., foliated by circles), 
or hyperbolic (i.e., admitting a complete Riemannian metric with constant sectional curvature $-1$). 
See \cite{S} for a survey. 
Thus, exceptional surgeries are also divided into three types; 
reducible (i.e., yielding a reducible manifold), 
toroidal (i.e., yielding a toroidal manifold), or 
Seifert fibered (i.e., yielding a Seifert fibered manifold).

\subsection{Families of knots}
We here introduce some notions for knots that we use in this paper. 


A \textit{bridge index} of a knot in $S^3$ is defined as 
the minimal number of local maxima (or local minima) up to ambient isotopy. 
Thus, a knot with bridge index $2$ is called a \textit{two-bridge knot}. 
Since two-bridge knots are alternating, a natural consequence of Menasco's work 
in \cite{Menasco1984} is that a two-bridge knot is hyperbolic unless it is a $(2,p)$-torus knot.


We now recall some standard notation and terminology regarding arborescent knots. 
See \cite{Wu1996} for full details. 
By a \textit{tangle}, we mean 
a pair with a 3-ball and properly embedded 1-manifolds. 
From two arcs of rational slope 
drawn on the boundary of a pillowcase-shaped 3-ball, 
one can obtain a tangle, which is called a \textit{rational tangle}. 
A tangle obtained by 
putting rational tangles together in a horizontal way 
is called a \textit{Montesinos tangle}. 
An \textit{arborescent tangle} is then defined as 
a tangle that can be obtained by summing several Montesinos tangles together in an arbitrary order. 

Suppose that a knot $K$ in $S^3$ is obtained by closing a tangle $T$. 
If $T$ is a Montesinos tangle, then we call $K$ a \textit{Montesinos knot}, and 
if $T$ is an arborescent tangle, then we call $K$ an \textit{arborescent knot}. 

The number of rational tangles 
forming the corresponding Montesinos tangle 
is called the \textit{length} of the Montesinos knot. 
It is seen that prime Montesinos knots with length at most two are all two-bridge knots. 
Thus, they are all alternating. 
On the other hand, Montesinos knots of length at least three are generally not alternating. 

In \cite{Wu1996}, Wu divided all arborescent knots into three types: 
 \textit{type I} knots - two-bridge knots or Montesinos knots of length 3, 
 \textit{type II} knots - the union of two Montesinos tangles, each of which is formed by two rational tangles 
corresponding to $1/2$ and a non-integer, and
all the other arborescent knots are \textit{type III}.


We denote by $M(r_1,r_2,\dots,r_n)$ 
a Montesinos knot constructed from rational tangles 
corresponding to rational numbers $r_1,r_2,\dots,r_n$. 
In particular, $M(1/q_1, 1/q_2, \dots, 1/q_n)$ with 
integers $q_1, q_2, \dots, q_n$ is called a \textit{pretzel knot} of $n$-strands.


\section{Proof of Corollary \ref{Cor}}\label{sec.3}

Let $K$ be a hyperbolic alternating knot in $S^3$. 
Suppose that the surgered manifold $K(r)$ is non-hyperbolic for some rational number $r$. 
As we recall above, we see that $K(r)$ is reducible, toroidal or Seifert fibered. 

First, $r$ must be an integer by the first author in \cite[Theorem 1.1]{Ichihara2008AGT}. 

Also, $K(r)$ must be irreducible by Menasco and Thistlethwaite \cite[Corollary 1.1]{MenascoThistlethwaite1992}. 

Suppose that $K(r)$ is toroidal. 
Then $K(r)$ is not Seifert fibered, 
shown by the first author with Jong \cite{IchiharaJong2013}. 
Moreover, as a consequence of the argument used 
in the classification of toroidal surgeries on alternating knots 
obtained by Patton \cite{Patton1995} and 
Boyer and Zhang \cite[Lemma 3.1]{BoyerZhang1997}, 
we see that $K$ and $r$ is a pair 
listed in the statement of Corollary~\ref{Cor}. 
Alternatively, this observation also follows from our Theorem \ref{MainThm}, 
together with the classifications of 
toroidal surgeries on two-bridge knots and other arborescent knots 
obtained by Brittenham and Wu \cite{BrittenhamWu2001}
and Wu \cite{Wu2011a}, \cite{Wu2011b}. 

Suppose that $K(r)$ is Seifert fibered. 
As noted above, by \cite{IchiharaJong2013}, $K(r)$ must be small Seifert fibered. 
Also, as shown in \cite{DelmanRoberts1999}, $K(r)$ has infinite fundamental group. 
Moreover, by Theorem \ref{MainThm}, $K$ is an arborescent knot. 
Any hyperbolic arborescent knot of type II or type III cannot 
admit a small Seifert fibered surgery by Wu in \cite{Wu1996}, \cite{Wu2011b}. 
Thus, $K$ must be an arborescent knot of type I, and so $K$ is either
a two-bridge knot or a Montesinos knot of length 3. 

If $K$ is a two-bridge knot, then $K$ is either 
the figure-eight knot and $r= \pm 1 , \pm 2 , \pm 3$ or 
a twist knot $K_{[2n,\pm 2]}$ with $|n| > 1$ and $r= \mp 1, \mp 2, \mp 3$ 
as claimed in the corollary by the result of Brittenham and Wu \cite{BrittenhamWu2001}. 

For Montesinos knots of length 3, 
except the particular family of Montesinos knots $M (-1/2, 2/5, 1/(2q + 1))$ with $q \ge 5$, 
Meier obtained a complete classification of exceptional surgeries in \cite{Meier}. 
Due to the classification, we see that all the Montesinos knots of length three admitting  small Seifert fibered surgeries are non-alternating as follows. 
Each of the knots can be checked that 
it admits reduced Montesinos diagrams which is non-alternating. 
Then the diagram is a minimal diagram, 
since if a Montesinos link admits 
an $n$-crossing reduced Montesinos diagram, 
then it cannot be projected with fewer than $n$ crossings, 
as shown in \cite[Theorem 10]{LickorichThistlethwaite1988}. 
However, a non-alternating projection of a prime alternating link 
cannot be minimal \cite[Theorem B]{Murasugi1987}, 
and so, the knots considered above are non-alternating. 
Finally we see that the Montesinos knots $M (-1/2, 2/5, 1/(2q + 1))$ with $q \ge 5$ are non-alternating in the same way. 
This completes the proof of Corollary \ref{Cor}. 
\qed

\bigskip

In Appendix \ref{ape.B}, 
we show that the Montesinos knots $M (-1/2, 2/5, 1/(2q + 1))$ with $q \ge 5$ 
actually have no non-trivial exceptional surgeries.

\section{Proof of Theorem~\ref{MainThm}}\label{sec.4}

Let $K$ be a hyperbolic alternating knot in the $3$-sphere. 
Suppose that $K$ admits an exceptional surgery, i.e., 
suppose that the surgered manifold $K(r)$ is non-hyperbolic for some integer $r$. 

First the following lemma, essentially due to Lackenby in \cite[Theorem~5.1]{lack}, 
shows that $K$ is not ``sufficiently complicated''. 

\begin{lemma}\label{lem1}
If a hyperbolic alternating knot $K$ has 
a connected prime alternating diagram $D$ satisfying $t(D) \ge 9$, 
then $K$ admits no non-trivial exceptional surgeries. 
\end{lemma}

In fact, Lackenby showed in \cite[Theorem 5.1]{lack} 
that the knots satisfying the assumption above have no ``non-hyperbolike'' surgeries. 
Then the Perelman's affirmative solution to the Geometrization Conjecture 
guarantees that ``non-hyperbolike'' is equivalent to  non-hyperbolic, 
i.e., exceptional in this context. 

Here we recall terminology used in the lemma above. 
Let $D$ be a connected alternating diagram of a knot in $S^3$, 
which we view as a $4$-valent graph embedded in $S^2$, 
equipped with ``under-over" crossing information.
Then $D$ is called \textit{prime} 
if each simple closed curve in $S^2$ intersecting $D$ transversely 
in two points divides $S^2$ into two discs, 
one of which contains no crossings of $D$. 
The \textit{twist number} of the diagram $D$, 
denoted by $t(D)$, is defined as 
the number of \textit{twists}, which are either maximal connected collections of bigon regions in $D$ arranged in a row 
or isolated crossings adjacent to no bigon regions.

In the case where $t(D) \le 8$, we first have the following: 

\begin{lemma}\label{lem42}
If a hyperbolic alternating knot $K$ in $S^3$ has 
a connected prime alternating diagram $D$ satisfying $t(D) \le 8$, 
then either $K$ is an arborescent knot or 
$K$ has a connected prime alternating diagram $D$ satisfying $6 \le t(D) \le 8$, 
and is obtained from one of the 9 plane graphs 
illustrated in Figure \ref{fig.graphs} by substituting one of the 4 tangles 
illustrated in Figure \ref{fig.tangles} to each of the fat vertices in the graphs, 
and performing twisting on all the augmented circles. 
\end{lemma}

Here we mean by an \textit{augmented circle} 
an unknotted component which encircles a crossing or a pair of parallel two strands in a given diagram.
Also we say that a knot is obtained by \textit{twisting on an augmented circle} of a link 
if it is obtained by performing $1/q$-surgery on the component ($q \in \mathbb{Z} \setminus \{ 0 \}$). 

\begin{figure}[htb]
\includegraphics[bb= 0 0 1482 1108 , width=\textwidth]{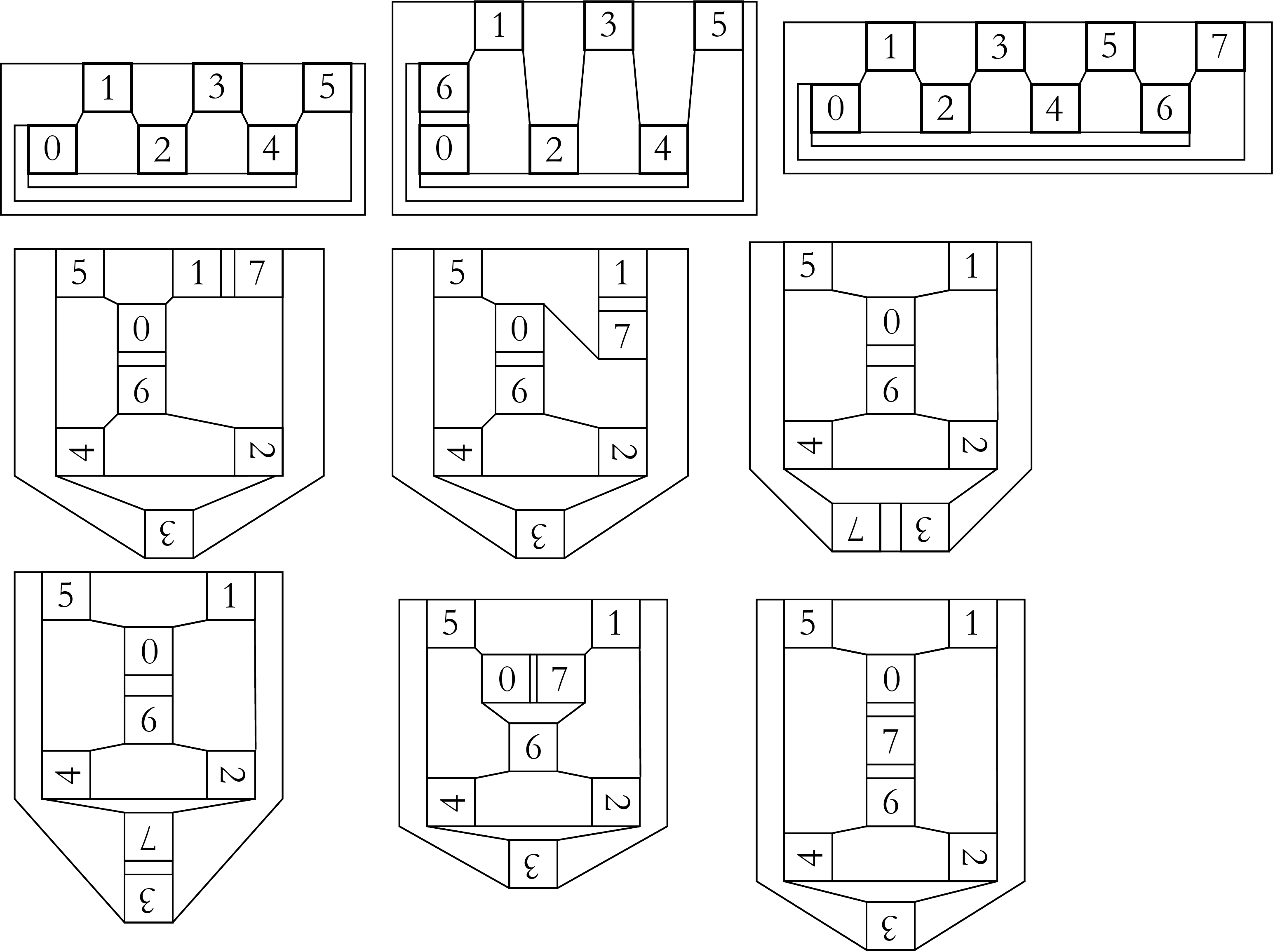}
\caption{9 plane graphs.}\label{fig.graphs}
\bigskip
\includegraphics[bb = 0 0 514 107 , scale = 0.6]{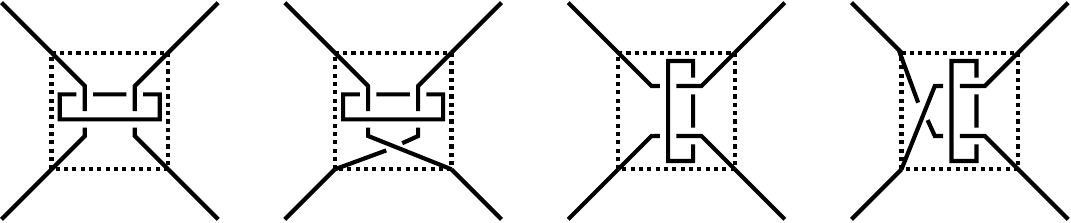}
\caption{4 tangles.}\label{fig.tangles}
\end{figure}

Note that among the graphs in Figure \ref{fig.graphs}, 
the top 3 graphs that will be named $G_6, G_7, G_8^s$ are simple and 3-connected. 
Here a graph is said to be \textit{simple} if it does not have any multi edge or self loop, and 
\textit{3-connected} if there does not exist a set of $2$ vertices whose removal is disconnected.
 
\begin{proof}[Proof of Lemma \ref{lem42}]
Let $K$ be a hyperbolic alternating knot in $S^3$, and $D$ a connected prime alternating diagram of $K$. 
Suppose that $K$ is not an arborescent knot and $t(D) \le 8$. 

As demonstrated in \cite[Section 5]{lack}, 
$D$ is obtained from some regular 4-valent plane graph with $t(D)$ vertices by replacing all of its vertices with twists. 
It is equivalent to say that $D$ is obtained from 
some regular 4-valent plane graph by substituting one of the 4 tangles 
illustrated in Figure \ref{fig.tangles} to each of the fat vertices in the graphs, 
and performing twisting on all the augmented circles. 
Now it suffices to show that, to obtain $D$, 
we only consider the 9 plane graphs depicted in Figure \ref{fig.graphs}. 

First we show that $t(D) \ge 6$. 
Suppose for a contrary that $t(D) \le 5$. 
Then $D$ is obtained from some regular 4-valent plane graph with at most 5 vertices. 
It is well-known that such a plane graph always has a complementary bigon. 
See \cite{Brinkmann_etal2005} for example. 
By collapsing a bigon to a `fat' vertex, we have a new regular 4-valent plane graph with fewer vertices. 
Then the original diagram $D$ is obtained by replacing all of its vertices with rational tangles or Montesinos tangles. 
Repeating this procedure, we have a regular 4-valent plane graph with a single vertex, 
from which we reconstruct the original diagram 
by replacing the vertex with an arborescent tangle. 
This means that the original $D$ must represent an arborescent knot, 
contradicting the assumption that $K$ is not an arborescent knot. 

Next suppose that $t(D) = 6$. 
If $D$ has a complementary bigon on the projection plane, 
then, in the same way as above, it is shown that $D$ must represent an arborescent knot, contradicting the assumption. 
Thus $D$ can admit no complementary bigons. 
Again, for example by \cite{Brinkmann_etal2005}, 
it is known that there is exactly one regular 4-valent plane graph, say $G_6$, with 6 vertices without complementary bigons, 
which is depicted at the left of the top row in Figure \ref{fig.graphs}. 

Next suppose that $t(D) = 7$. 
Again, for example by \cite{Brinkmann_etal2005}, 
it is known that there is no regular 4-valent plane graph with 7 vertices without complementary bigons. 
This means that, under the assumption that $D$ does not represent an arborescent knot, 
$D$ is obtained from $G_6$ by replacing a vertex with a vertical or a horizontal bigon. 
Since $G_6$ is homogeneous, 
i.e. 
the graph automorphism group acts transitively on the set of vertices,
it suffices to consider only one graph, say $G_7$, 
which is depicted at the middle of the top row in Figure \ref{fig.graphs}. 

Finally suppose that $t(D) = 8$. 
Again, for example by \cite{Brinkmann_etal2005}, 
it is known that there is exactly one regular 4-valent plane graph with 8 vertices without complementary bigons, 
say $G_8^s$, which is depicted at the right of the top row in Figure \ref{fig.graphs}. 
The other possibility is that $D$ is obtained from $G_6$ 
by twice repetition of replacing a vertex with a vertical or a horizontal bigon. 
As mentioned above, $G_7$ is unique up to symmetry.
Hence for this case we only need to add bigon to $G_7$.
By the symmetry of $G_7$ depicted in Figure \ref{fig.addbigon},
we see that there are 6 distinct ways to add bigons.

\begin{figure}[htb]
\includegraphics[bb = 0 0 1055 507 , scale = 0.3]{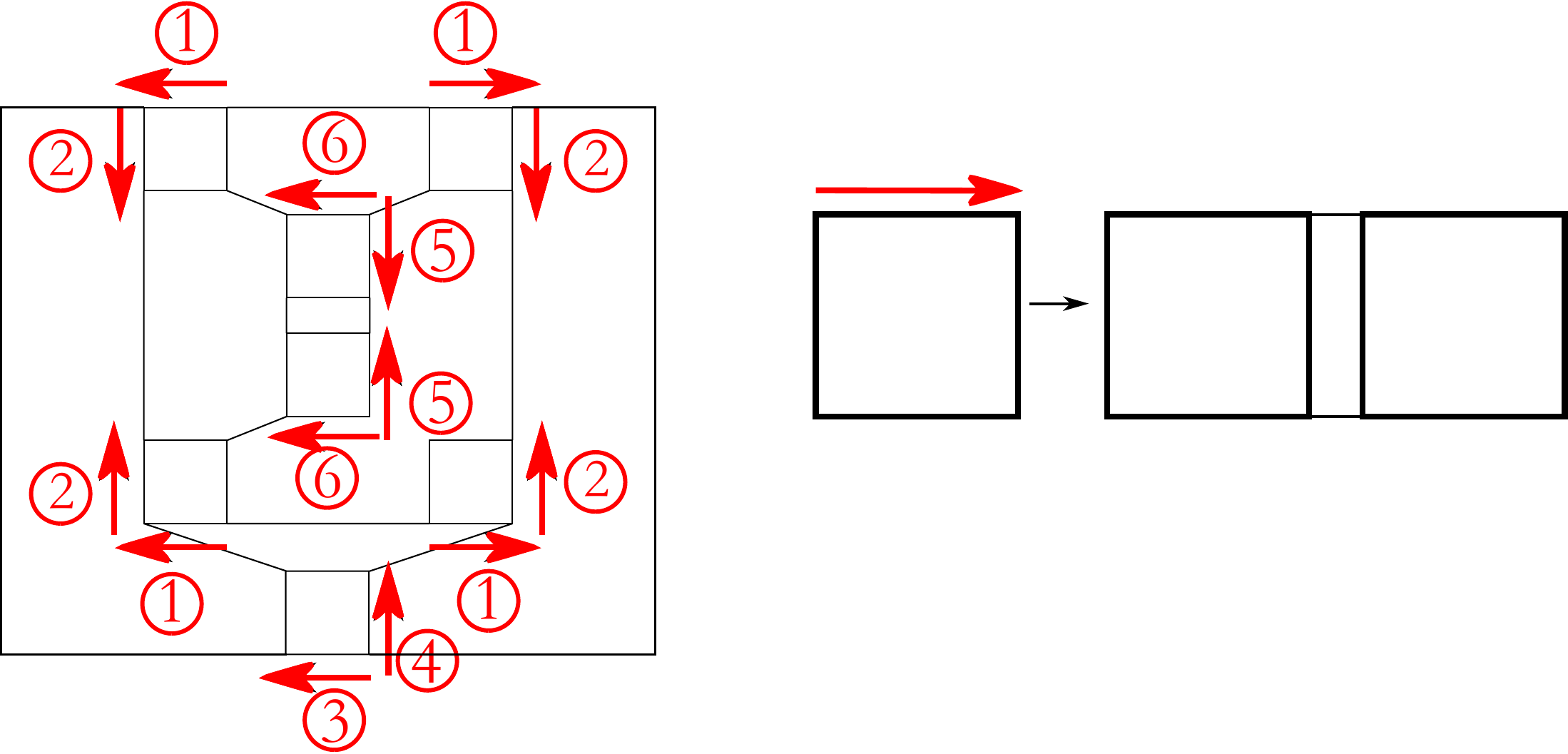}
\caption{(left) 6 places and directions (up to symmetry) to add a bigon, (right) the meaning of arrows.}\label{fig.addbigon}
\end{figure}

Those 6 graphs obtained by adding bigons to $G_7$, say $G_8^1 , \dots, G_8^6$, are
depicted at the middle and the bottom rows in Figure \ref{fig.graphs}. 

This completes the proof of the lemma. 
\end{proof}

The number of the links obtained above is naively estimated as $4^6 + 4^7 + 7 \cdot 4^8 = 479232$. 
The next lemma efficiently reduces the number of links we have to consider. 

\begin{lemma}\label{lem43}
Let $K$ be a hyperbolic alternating non-arborescent knot admitting a diagram $D$ such that $6 \leq t(D) \leq 8$. 
Suppose that $K$ admits a non-trivial exceptional surgery. 
Then there are 30404 hyperbolic links with augmented circles, 
which are constructed by substituting one of the 4 tangles in Figure \ref{fig.tangles} 
to each of the fat vertices of one of the 9 plane graphs in Figure \ref{fig.graphs}, 
such that 
$K$ is obtained from one of the links by performing twisting on the augmented circles. 
\end{lemma}

A proof of this lemma is given in the next section, which is computer-aided. 
The detailed explanation of the key piece of our code \texttt{twistLink.py}'s 
used in the proof is included.
We actually have 9 files; 
\texttt{twistLink6.py}, \texttt{twistLink7.py}, 
\texttt{twistLink8\_1.py}, $\dots$ , \texttt{twistLink8\_7.py}, one corresponds to a graph
in Figure \ref{fig.graphs},
and as a set we call them \texttt{twistLink.py}'s. 

\begin{lemma}\label{lem44}
Let $K$ be a knot obtained from a Dehn surgery on $L$, one of the 30404 augment links in $S^3$ obtained in Lemma \ref{lem43}, 
such that the surgery corresponds to twisting along the unknotted components of $L$.
If $K$ is alternating, then $K$ is hyperbolic and admits no non-trivial exceptional surgeries.
\end{lemma}

This lemma is proved by super-computer calculations, 
mainly by applying the program named \texttt{fef.py}. 
\S 6 is devoted to give detailed explanations of the code \texttt{fef.py}.

Now we are in the position to prove our main result. 

\begin{proof}[Proof of Theorem \ref{MainThm}]
Let $K$ be a hyperbolic alternating knot in $S^3$. 
Suppose that $K$ admits a non-trivial exceptional surgery. 
We show that $K$ is equivalent to an arborescent knot.

Since $K$ is hyperbolic, $K$ is prime by \cite[Corollary 2.2]{Thurston1982}. 
Let $D$ be an alternating diagram of the knot $K$. 
Then $D$ is connected since $K$ is a knot (not a link). 
Since $K$ is prime, the diagram $D$ is also prime by \cite[Theorem (b)]{Menasco1984}. 

Now we consider the twist number of the diagram $D$. 
By Lemma \ref{lem1}, if $t(D) \ge 9$, then $K$ admits no non-trivial exceptional surgeries.
Then, by combining Lemmas \ref{lem42}, \ref{lem43}, \ref{lem44}, 
any hyperbolic alternating non-arborescent knot in $S^3$ admits no non-trivial exceptional surgeries.
\end{proof}

\section{Reducing the number of augmented links and components}\label{sec.symmetry}

In this section, we give a proof of Lemma \ref{lem43}. 
Throughout the section, assume $K$ is a hyperbolic alternating non-arborescent knot $K$ 
with a diagram $D$ of twist number $t(D)$ satisfying $6\leq t(D)\leq 8$ admitting a non-trivial exceptional surgery. 
Then we will show that there are 30404 hyperbolic links with augmented circles such that 
$K$ is obtained from one of the links by performing twisting on the augmented circles. 

The outline of this section is as follows. 
By Lemma \ref{lem42}, each of the links to be considered is obtained from one of the 9 plane graphs 
illustrated in Figures \ref{fig.graphs} by substituting one of the 4 tangles 
illustrated in Figure \ref{fig.tangles} to each of the fat vertices in the graphs, 
and performing twisting on all the augmented circles. 
In \S \ref{subsec.settings}, we describe a method to encode and enumerate these links as sequences of elements in $\{0,1,2,3\}$.
In \S \ref{subsec.conditions}, we explain and enforce conditions on the sequences in order 
to reduce the number of links we need to investigate. 
To further reduce computation time 
needed to prove Lemma \ref{lem44}, we will give a condition 
to reduce the number of components of the links so obtained. 
Our key ingredient is Lemma \ref{cor.Wu} based on the study 
of genuine laminations which remain genuine after any non-trivial Dehn surgery. 
This method extends the result of Wu in \cite{Wu2012}. 

Together with considerations of our restrictions, 
we implemented our procedure as a set of files \texttt{twistLink.py}'s.
All files used and data of outputs are available at \cite{Webpage}. 
In \S \ref{subsec.codes}, we will explain these files.
Note that our complete program has two parts. This part, used to prove Lemma \ref{lem43}, generates triangulation files of SnapPea. These files are then analyzed in the proof of Lemma \ref{lem44}, which will be explained in the next section.

\subsection{Settings}\label{subsec.settings}

We want to obtain the links with augmented circles such that 
any hyperbolic alternating non-arborescent knot 
with diagram $D$ of twist number $t(D)$ satisfying $6\leq t(D)\leq 8$ which admits a non-trivial exceptional surgery 
is obtained from one of the links by performing twisting on the augmented circles. 
By Lemma \ref{lem42}, such links are obtained from one of the 9 graphs $G_6, G_7, G_8^s, G_8^1 , \dots, G_8^6$ in Figure \ref{fig.graphs}. 
Note that each square with a figure in it is a vertex. 
We will call the square with $i$ in it the $i$-th square.

We first explain how we relate such a link to a sequence of $\{0,1,2,3\}$ of length $l$, where $l=6$, $7$, or $8$ depending on the graph.  

Let $\{a_i\}_{i=0}^{l-1}$ be one of such sequences.
Then we fill the $i$-th square with a tangle according to the correspondence
which is depicted in Figure \ref{fig.fill_tangle}. 
Namely, we fill $i$-th square with one of the two string tangle with an augmented circle such that the two strands connect 
\begin{itemize}
\item (nw,sw) and (ne,se) respectively if $a_i =0$, 
\item (nw,ne) and (sw,se) respectively if $a_i =2$, or
\item (nw,se) and (ne,sw) respectively if $a_i =1,\mbox{ or } 3$,
\end{itemize}
and the augmented circle is
\begin{itemize}
\item horizontal if $a_i\in\{0,1\}$, or
\item vertical if $a_i\in\{2,3\}$.
\end{itemize}
Here nw corresponds to the north west corner of the square and we define ne, sw, and se similarly.
Note that the orientation that determines nw, ne, sw, and se is determined by the orientation of the figure in it.
We remark that there is an ambiguity of the sign of crossings when $a_i \equiv 1 \mbox{ mod } 2$. 
It will be explained in Remark \ref{rmk.sing_crossing} how to choose either of them.

\begin{figure}[htb]
\includegraphics[bb= 0 0 689 123 , scale = 0.5]{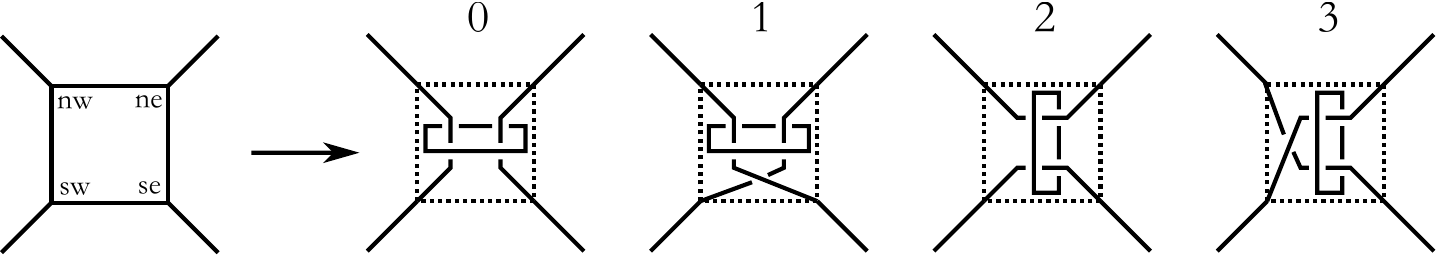}
\caption{Fill tangle.}\label{fig.fill_tangle}
\end{figure}

We drew in Figure 5 
the link that corresponds to the sequence $021213$ as an example.

\begin{figure}[htb]
\includegraphics[width=.8\textwidth]{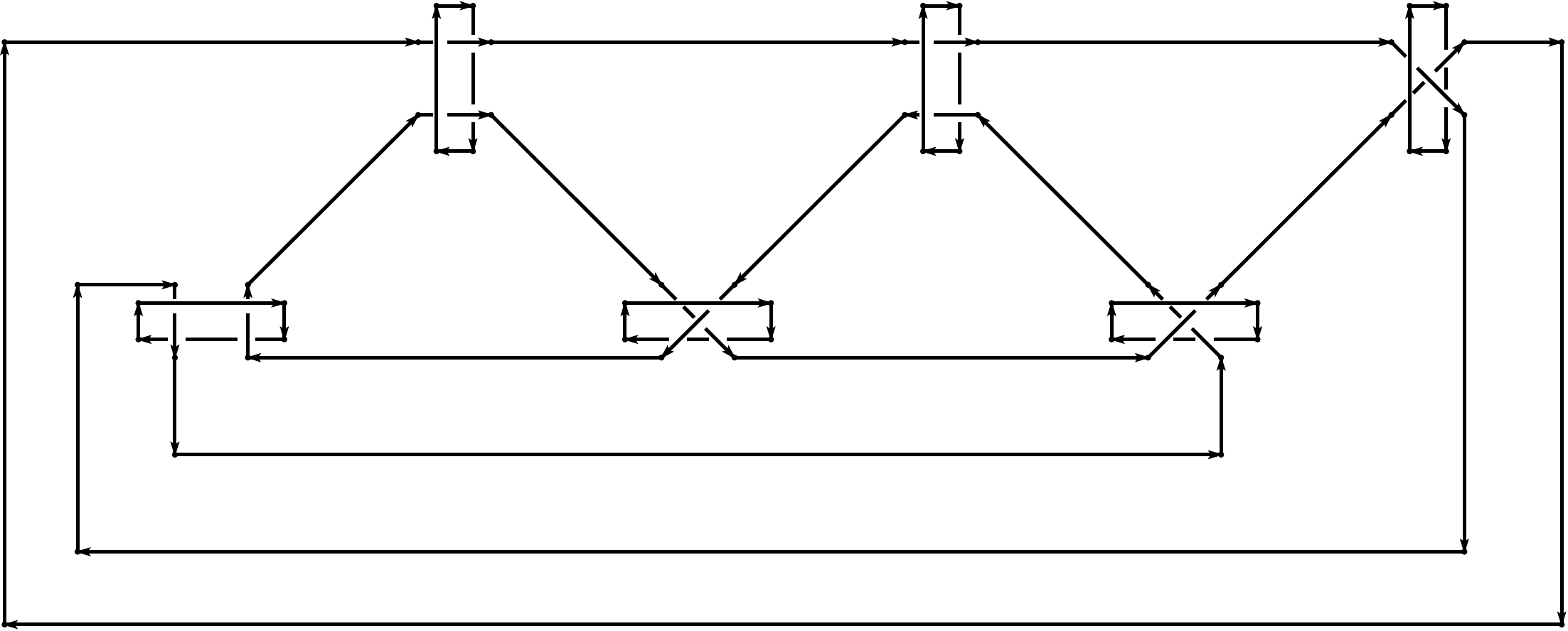}
\label{fig.example6twist}
\caption{The link corresponding to 021213.}
\end{figure}

Since performing surgery on an augmented circle with slope $-1/p$ (resp. $1/p$) corresponds 
adding positive (resp. negative) $p$ full twist to the knot components running through the augmented circle, 
by varying the signs and $p$, we can enumerate all links with diagrams we must consider. 
However, first we enumerate the augmented links by associating each link to a sequence $\{a_i\}_{i=0}^{l-1}$,
where $l$ is the number of augmented circles.  We define a {\em knot component} as a component of the link which is not an augmented circle.

\subsection{Conditions}\label{subsec.conditions}

We now describe conditions which will reduce the number of sequences we need to consider.

\subsubsection{Alternating knots and Twist number}\label{sect:conditions}

Since we are only interested in {\em alternating knots}, we have some constraints.
First, we only need to deal with sequences whose corresponding links have one connected knot component, 
for otherwise after twisting we get a link rather than a knot. 

\begin{condition}\label{cond.knot}
We only consider sequences, each of whose related link has one connected knot component.
\end{condition}

Given an alternating diagram the mirror image of the diagram will also be alternating.
Therefore after fixing the sign of the surgery slopes along the augmented circles,
we may assume that the sign on the $0$-th component is negative. 
This reduces the number of cases we need to consider by a factor of two.
We thus obtained the following conditions on slopes.

\begin{condition}\label{cond.alt}
For the link which is related to a sequence $\{a_i\}_{i=0}^{l-1}$, 
we consider the following conditions on surgery slopes;

\begin{itemize}
\item the slope of the augmented circle in 0-th square is $-1/p_0$ for some $p_0>0$,
\item the slopes of the other augmented circles are of type $1/p_i$ with $p_i\not= 0$, and 
	their signs are determined so that the resulting knot is alternating.
\end{itemize}
\end{condition}

We will only consider surgeries along slopes satisfying Condition \ref{cond.alt}.

\begin{remark}\label{rmk.sing_crossing}
There was an ambiguity of the sign of crossings when $a_i \equiv 1 \mbox{ mod } 2$. 
We choose the sign of such crossings 
so that we only need to look at slopes $1/p$ or $-1/p$ with $p$ ranging all positive integers. 
\end{remark}

Thus we generate a list of links with information of 
the number of augmented circles on which we perform surgery with negative slope, i.e. $-1/p$'s with positive $p$. 
This number is equal to the number of positive twists of the resulting alternating knots. 
For each link in the list we have generated, the $0$-th component is the knot component, 
$1$-st to $i$-th components are the augmented circles which will be surgered with negative slopes, 
and the rest will be surgered with positive slopes. 

\bigskip

For the case of 7 and 8 twists, equivalently, the case of length 7 and 8 sequences, 
not only requiring knot component to be connected, 
we also require that, after twisting, the twist number does not decrease. 
For example, in the length 7 case, we require either $a_0\in\{2,3\}$ or $a_6\in\{2,3\}$. 
For otherwise, it can readily be seen that any resulting knot after twisting has at most $6$ twists.
We have similar conditions for the length 8 cases. 
This occurs if there are \textit{parallel} augmented circles. 
Here augmented circles are said to be parallel if they are mutually isotopic in the complement of knot component. 
Thus we have the following condition. 

\begin{condition}\label{cond.paraaug}
The link which is related to a sequence $\{a_i\}_{i=0}^{l-1}$ has no pair of parallel augmented circles.
\end{condition}

\subsubsection{Symmetry}\label{subsect:symmetry}

Next we use symmetries of the plane graph $G_6, G_7, G_8^s,\allowbreak G_8^1 , \dots, G_8^6$ to reduce the number of sequences
to consider. 

We first discuss the graph $G_6$, which corresponds to the 6 twists case. For the remainder of this section,
we will denote a symmetry by its permutation on the set of tangle regions and when necessary add or subtract a number of twists. For example,
the symmetries of the graph $G_6$ are as follows:

\begin{figure}[htb]
\includegraphics[bb = 0 0 426 180 , scale = 0.3]{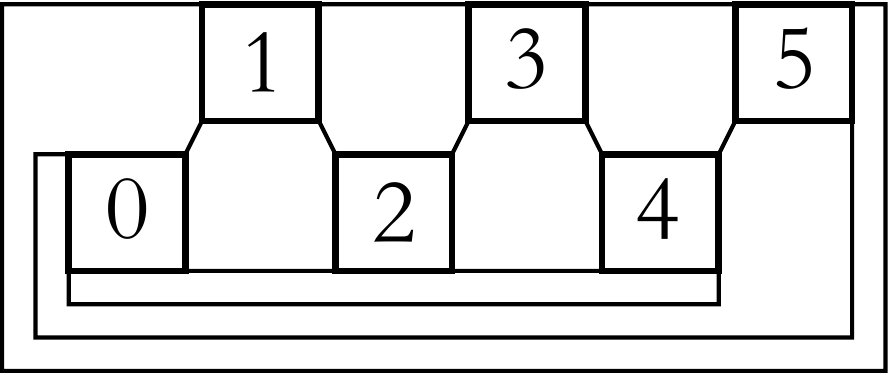}
\caption{The regular 4-valent simple plane graph with 6 vertices.}
\label{fig.6-twist}
\end{figure}

\begin{itemize}
\item $a_0a_1a_2a_3a_4a_5 \mapsto a_1a_2a_3a_4a_5a_0$,
\item $a_0a_1a_2a_3a_4a_5  \mapsto a_5a_4a_3a_2a_1a_0$.
\end{itemize} 

These symmetries are good enough to reduce the number of sequences in the list and are easy to implement. 
We implement the above procedure as \texttt{twistLink6.py} and by running it, we get a list with 185 links.

\bigskip

Next we consider the graph $G_7$ (Figure \ref{fig.7twist}). 

\begin{figure}[htb]
\includegraphics[bb = 0 0 426 252 , scale = 0.3]{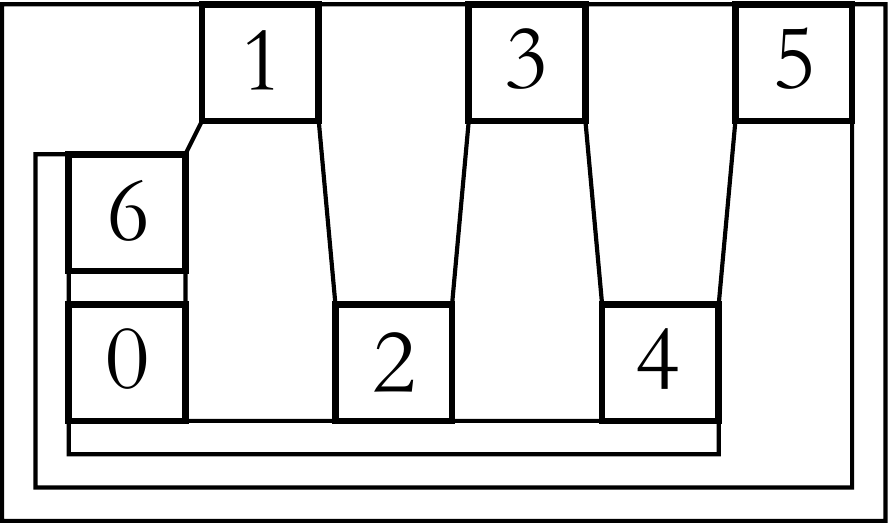}
\caption{The graph $G_7$ for 7 twists.}\label{fig.7twist}
\end{figure}

We will use the following symmetries. 
Here each element should be in $\{0,1,2,3\}$ and hence when we add $2$, it will be modulo 4.
Note that the symmetries we used here are the vertical bilateral symmetry and a $\pi$-rotation,
see Figure \ref{fig.7twist_sym}.
\begin{figure}[htb]
\includegraphics[bb= 0 0 442 370 , scale = 0.3]{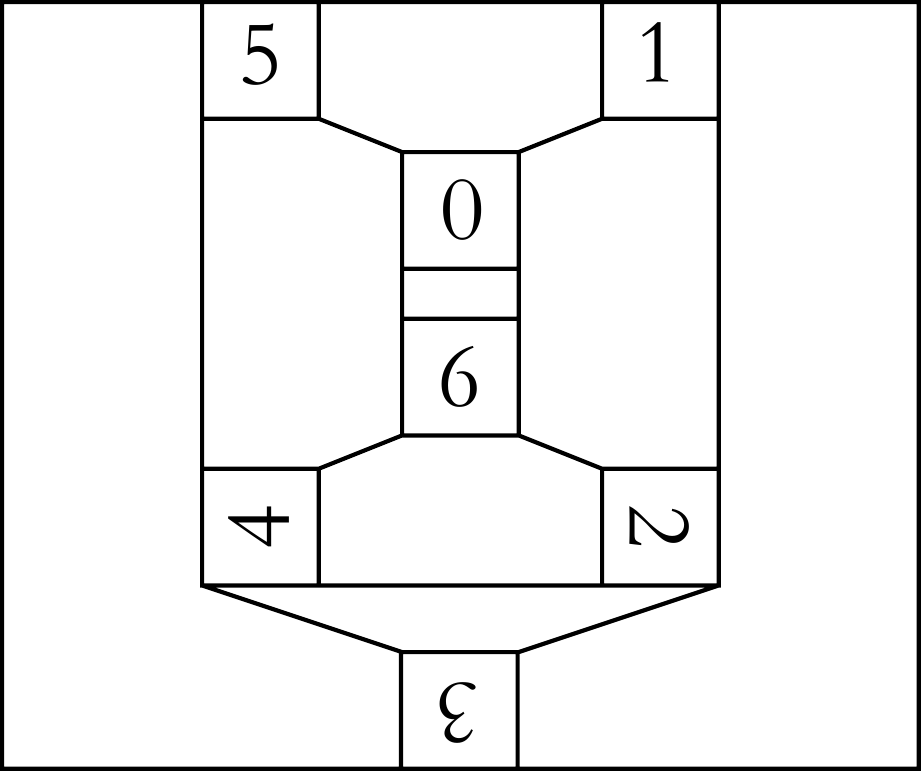}
\caption{$G_7$ has two bilateral symmetries.}\label{fig.7twist_sym}
\end{figure}

\begin{itemize}
\item $a_0a_1a_2a_3a_4a_5a_6 \mapsto a_6(a_2+2)(a_1+2)a_3(a_5+2)(a_4+2)a_0$,
\item $a_0a_1a_2a_3a_4a_5a_6  \mapsto a_6(a_4 + 2)(a_5 + 2)a_3(a_1 + 2)(a_2 + 2)a_0$.

\end{itemize} 

\bigskip

Finally we consider the graphs with 8 vertices. 
As shown in Lemma \ref{lem42}, we need to consider two types of graphs;
The unique simple plane graph $G_8^s$ (see Figure \ref{fig.8simple}), and $G_8^i$'s.
See Figure \ref{fig.g81}, \ref{fig.g82}, \ref{fig.g83}, \ref{fig.g84}, \ref{fig.g85}, and \ref{fig.g86} for pictures and symmetries of $G_8^i$'s.

\begin{figure}[htb]
\includegraphics[bb = 0 0 570 180 , scale = 0.3]{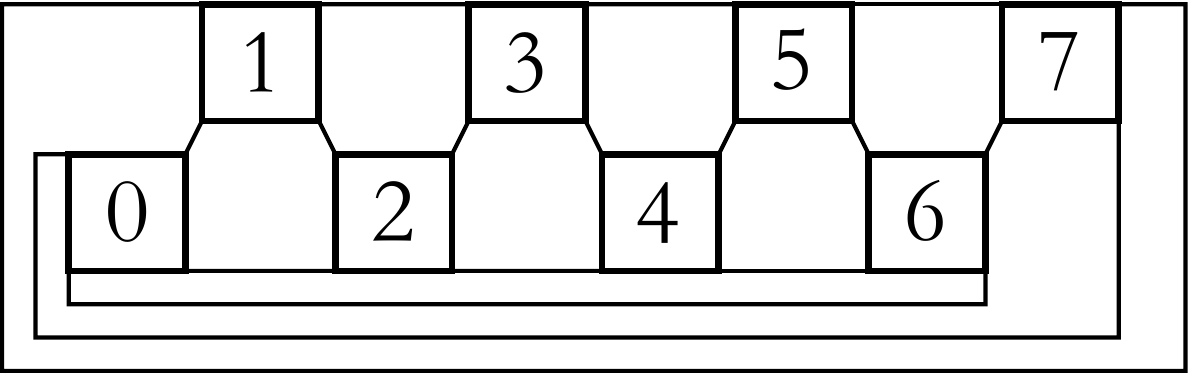}
\caption{The regular 4-valent simple planar graph with 8 vertices.}\label{fig.8simple}
\label{fig.8-twist}
\end{figure}
\begin{figure}[htb]
\includegraphics[bb = 0 0 362 362 , scale = 0.3]{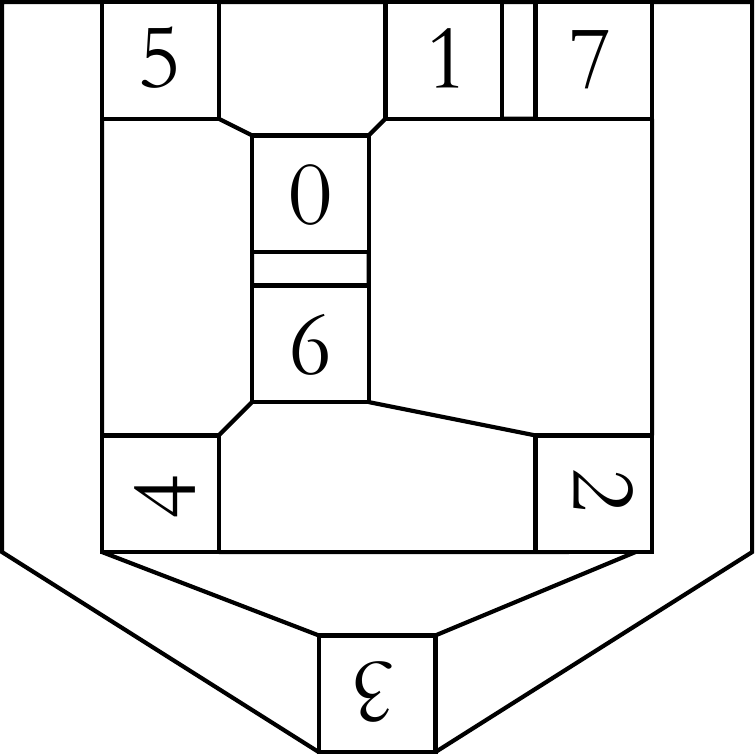}
\caption{$G_8^1$.}\label{fig.g81}
\end{figure}
\begin{figure}[htb]
\includegraphics[bb = 0 0 786 362 , scale = 0.3]{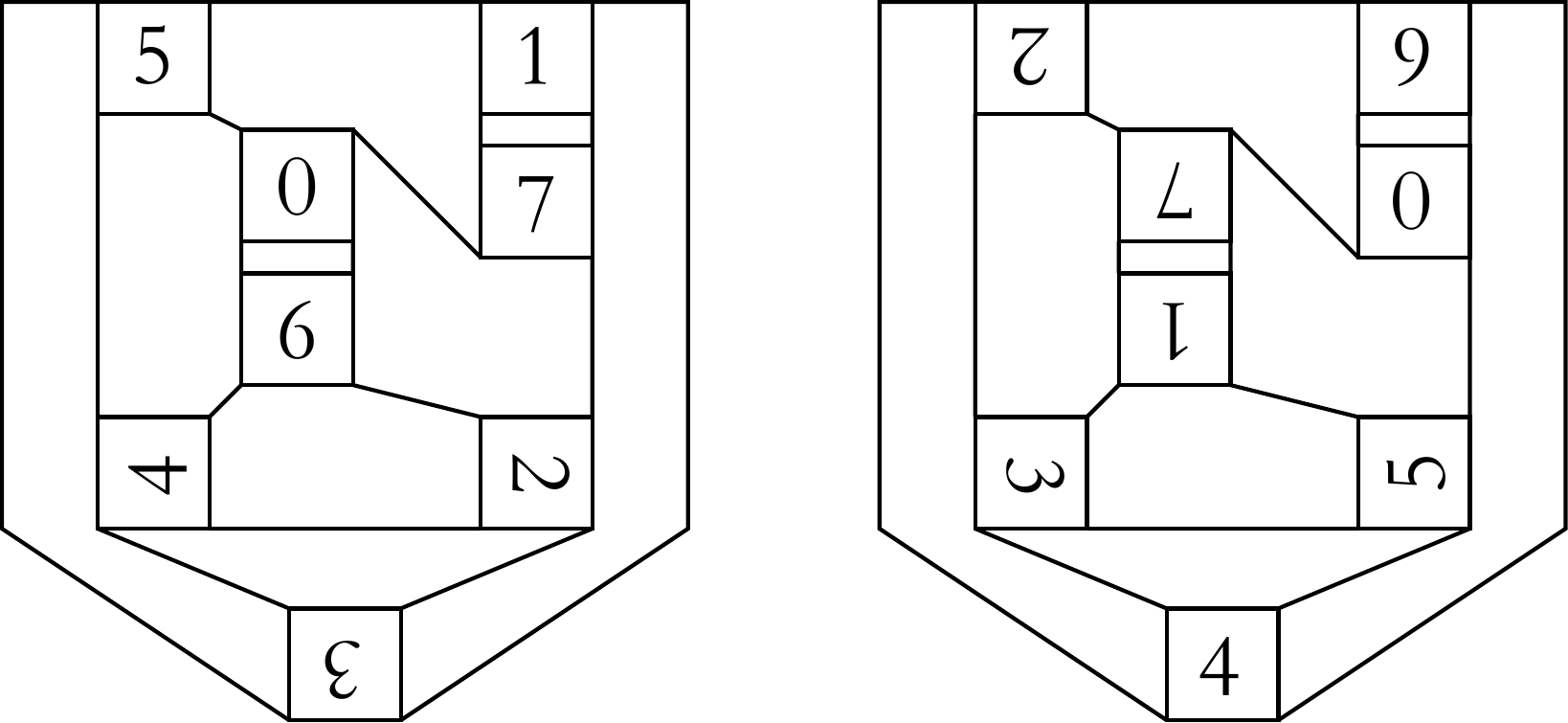}
\caption{$G^2_8$ (left) and the image under the action of its symmetry (right).}
\label{fig.g82}
\end{figure}
\begin{figure}[htb]
\includegraphics[bb = 0 0 706 362 , scale = 0.3]{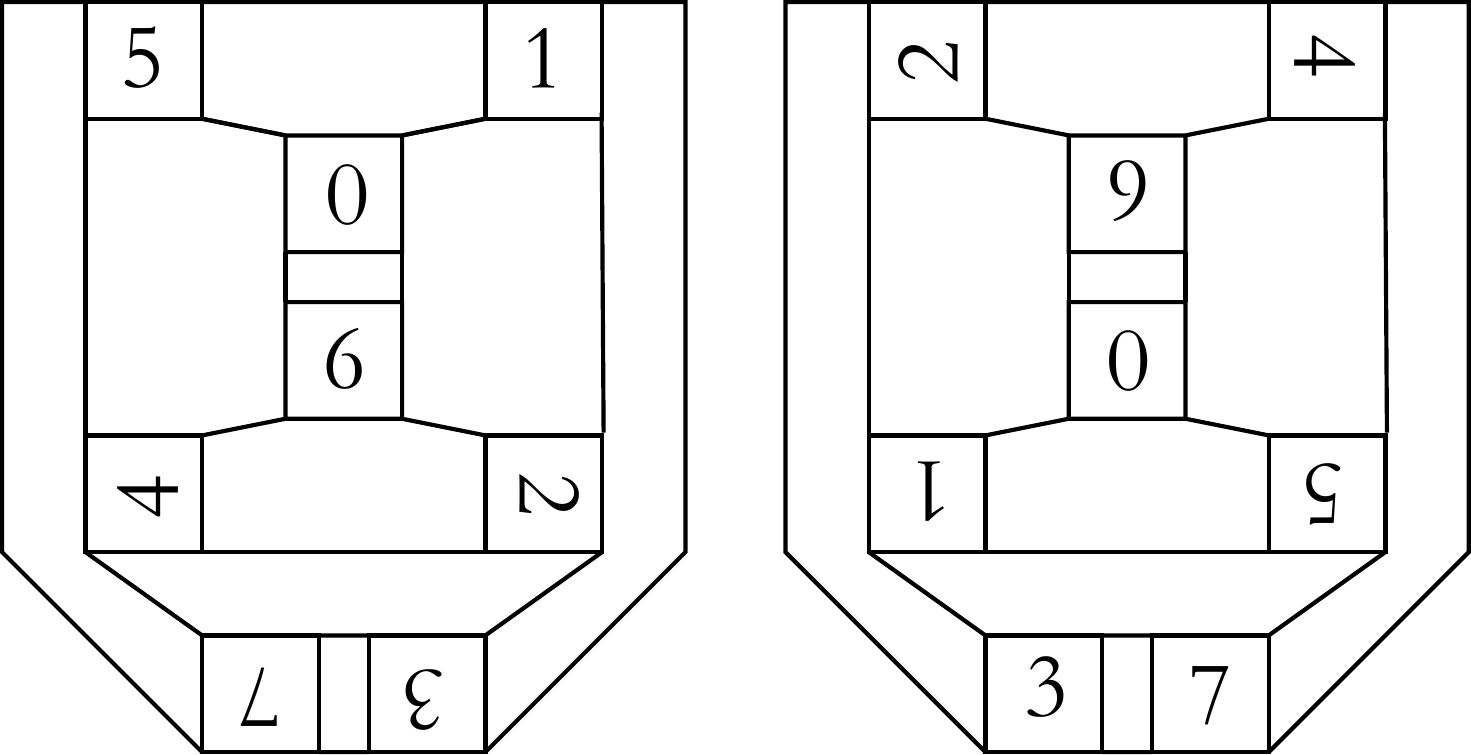}
\caption{$G_8^3$ (left) and the image under the action of its symmetry (right).}\label{fig.g83}
\end{figure}
\begin{figure}[htb]
\includegraphics[bb = 0 0 1066 410 , scale = 0.3]{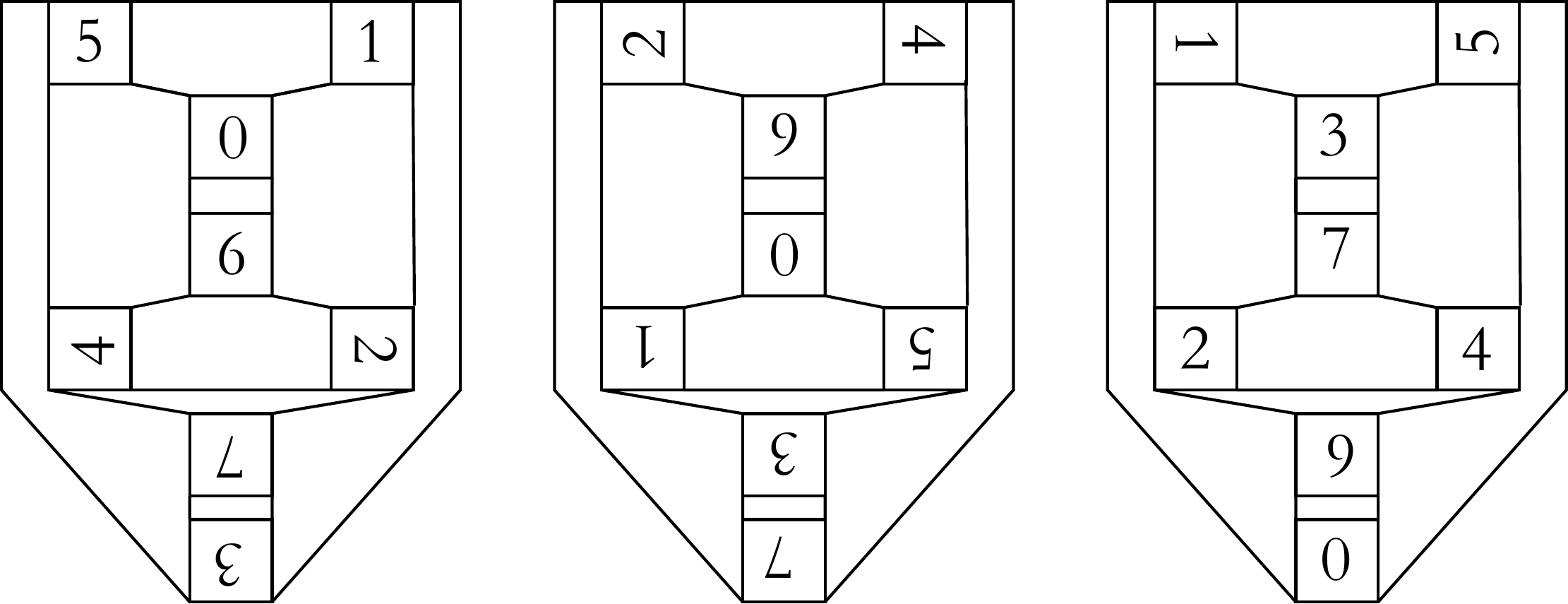}
\caption{$G_8^4$ (left) and the images under the action of its symmetry (middle) and (right).}\label{fig.g84}
\end{figure} 
\begin{figure}[htb]
\includegraphics[bb = 0 0 314 338 , scale = 0.3]{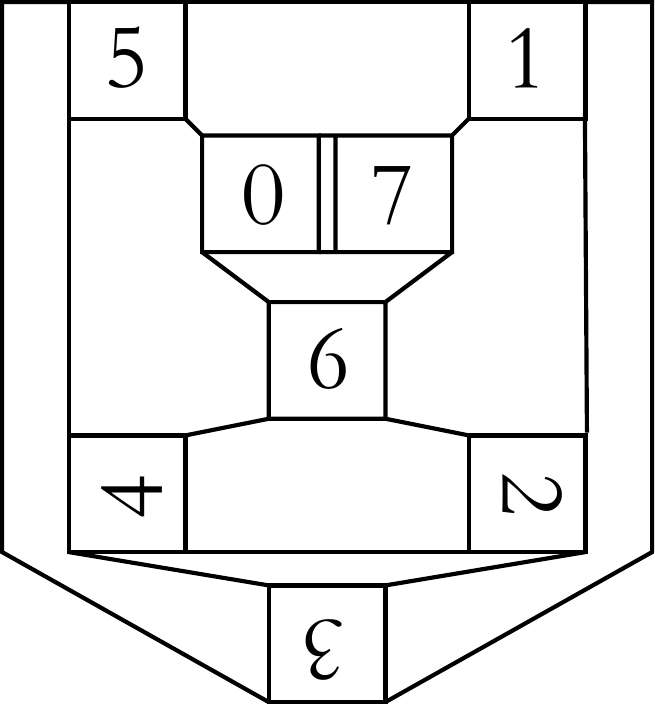}
\caption{$G_8^5$.}\label{fig.g85}
\end{figure}
\begin{figure}[htb]
\includegraphics[bb = 0 0 682 410 , scale = 0.3]{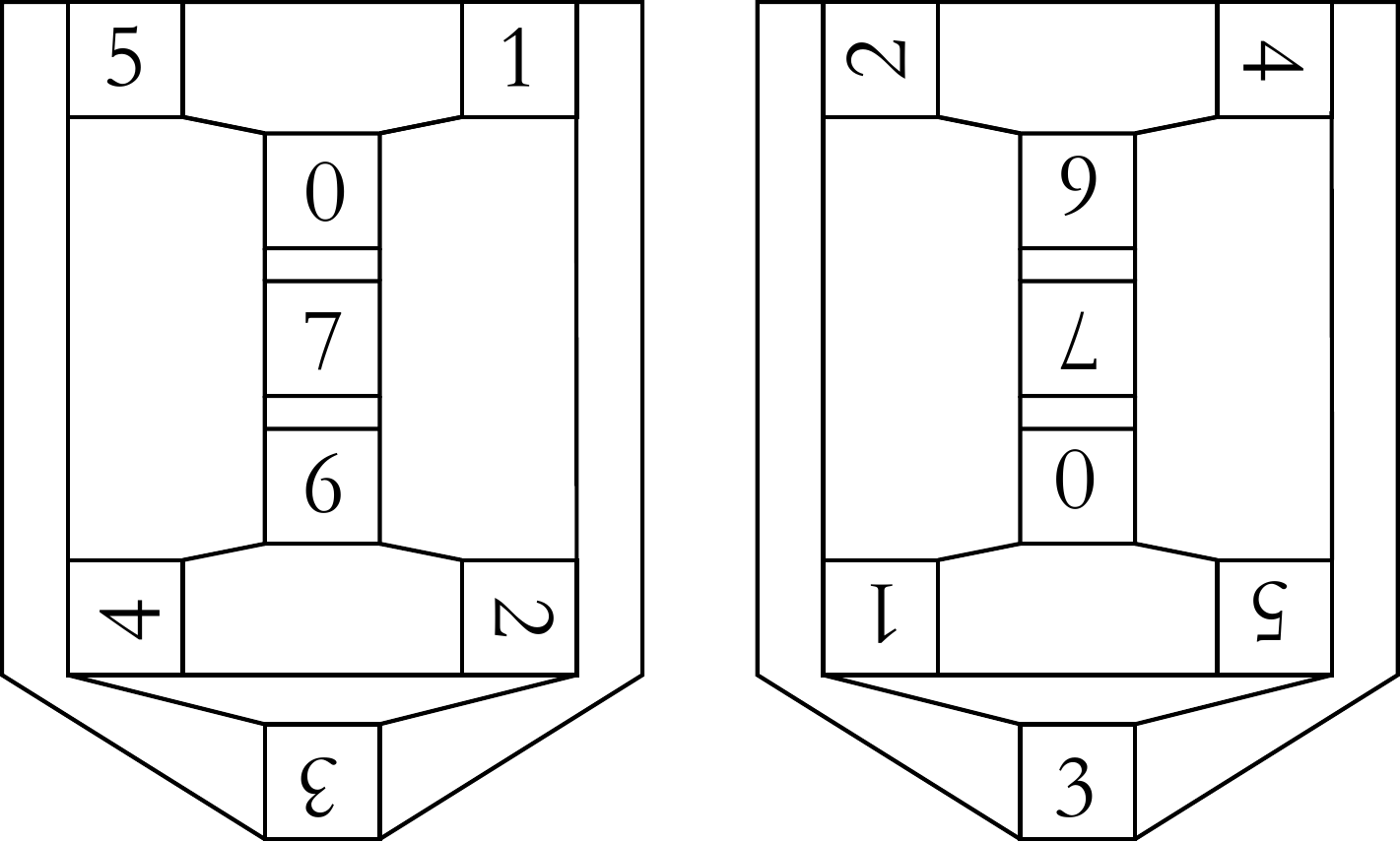}
\caption{$G_8^6$ (left) and the image under the action of its symmetry (right)}\label{fig.g86}
\end{figure}

We now summarize the symmetry which is depicted in the figures in terms of sequences.
Note that $G_8^3$, $G_8^4$, $G_8^5$, and $G_8^6$ have bilateral symmetries and
the symmetries following the mark ``*" are those corresponding to bilateral symmetries.
As we did for 7 twists case, the addition here is modulo~4.

\begin{enumerate}
\item $G_8^s$
\begin{itemize}
\item $a_0a_1a_2a_3a_4a_5a_6a_7 \mapsto a_1a_2a_3a_4a_5a_6a_7a_0$,
\item $a_0a_1a_2a_3a_4a_5a_6a_7  \mapsto a_7a_6a_5a_4a_3a_2a_1a_0$.
\item * $a_0a_1a_2a_3a_4a_5a_6a_7  \mapsto a_6a_5a_4a_3a_2a_1a_0a_7$.
\item * $a_0a_1a_2a_3a_4a_5a_6a_7  \mapsto a_0a_7a_6a_5a_4a_3a_2a_1$.
\end{itemize}
\item $G_8^1$
\begin{itemize}
\item No symmetry.
\end{itemize}
\item $G_8^2$
\begin{itemize}
\item $a_0a_1a_2a_3a_4a_5a_6a_7  \mapsto a_7a_6a_5a_4a_3a_2a_1a_0$.
\end{itemize}
\item $G_8^3$
\begin{itemize}
\item $a_0a_1a_2a_3a_4a_5a_6a_7 \mapsto a_6(a_4+2)(a_5+2)a_7(a_1+2)(a_2+2)a_0a_3$, and
\item * $a_0a_1a_2a_3a_4a_5a_6a_7  \mapsto a_0a_5a_4a_7a_2a_1a_6a_3$.    
\end{itemize}
\item $G_8^4$
\begin{itemize}
\item $a_0a_1a_2a_3a_4a_5a_6a_7 \mapsto a_6(a_4+2)(a_5+2)a_7(a_1+2)(a_2+2)a_0a_3$, 
\item $a_0a_1a_2a_3a_4a_5a_6a_7 \mapsto a_3(a_5+2)(a_4+2)a_0(a_2+2)(a_1+2)a_7a_6$, and
\item * $a_0a_1a_2a_3a_4a_5a_6a_7 \mapsto a_6(a_2+2)(a_1+2)a_7(a_5+2)(a_4+2)a_0a_3$.
\end{itemize}
\item $G_8^5$
\begin{itemize}
\item * $a_0a_1a_2a_3a_4a_5a_6a_7  \mapsto a_7a_5a_4a_3a_2a_1a_6a_0$.
\end{itemize}
\item $G_8^6$
\begin{itemize}
\item $a_0a_1a_2a_3a_4a_5a_6a_7 \mapsto a_6(a_4+2)(a_5+2)a_3(a_1+2)(a_2+2)a_0a_7$.
\item * $a_0a_1a_2a_3a_4a_5a_6a_7  \mapsto a_0a_5a_4a_3a_2a_1a_6a_7$.
\end{itemize}
\end{enumerate}

\subsection{Persistent genuine lamination}

Using Conditions \ref{cond.knot} and \ref{cond.paraaug} together with symmetries of the graph discussed above, 
we reduce the number of sequences, equivalently, the number of links we have to consider. 
However, in the proof of Lemma \ref{lem44}, the number of components of the links is crucial on computational time. 
The following lemma enables us to reduce the number of components for most of the links. 
This is an application of the result obtained by Wu in \cite{Wu2012}, which is of interest in its own right.

\begin{lemma}\label{cor.Wu}
Let $L_0$ be a link corresponding to a sequence satisfying
Conditions \ref{cond.knot} and \ref{cond.paraaug} for one of the 9 graphs in Figure \ref{fig.graphs}
as explained in \S \ref{subsec.settings}.
After assigning an orientation to the knot component of $L_0$,
suppose that the pair of segments on the knot component
passing through an augmented circle $A$ of $L_0$ are anti-parallel (see Figure \ref{fig.antipara}).
If we perform Dehn surgeries on $A$ satisfying Condition \ref{cond.alt} other than the ones corresponding to single full-twists, 
then any alternating knot obtained by twisting along the other augmented circles 
admits no exceptional surgeries, 
or it has a reduced alternating diagram 
with twist number less than the number of vertices of the plane graph 
used to construct $L_0$. 
\end{lemma}

\begin{figure}[htb]
\includegraphics[bb = 0 0 450 72 , scale = 0.75]{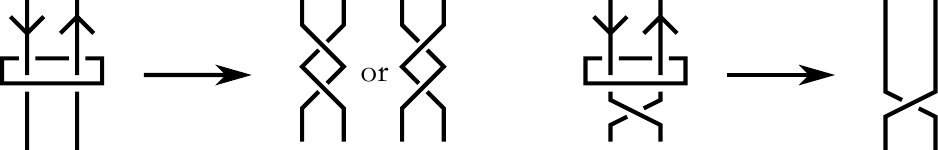}
\caption{Allowable twisting for anti parallel edges.}\label{fig.antipara}
\end{figure}

\begin{remark}
In Figure \ref{fig.antipara}, only one of the full-twists is shown in the case $A$ has a  crossing, namely the one with one crossing remaining, 
because we just perform Dehn surgeries on $A$ satisfying Condition \ref{cond.alt}. 
Please see Remark \ref{rmk.sing_crossing}. 
\end{remark}

\begin{proof}[Proof of Lemma \ref{cor.Wu}]
First we show that any Dehn surgeries on the alternating knots which we consider in the lemma 
yield manifolds containing essential laminations 
by using the following result obtained by Wu \cite[Corollary~6.9]{Wu2012}. 
We here omit the definition and properties of essential laminations. See \cite{Wu2012} for details. 

Let $L$ be a non-split oriented link, and $F$ a $\pi_1$-injective spanning surface of $L$. 
Take an arc $\alpha$ on $F$, and take 
a regular neighborhood $D$ of $\alpha$ embedded in $F$. 
Set up a coordinate on the boundary of $B = N ( \alpha )$ in $S^3$ 
so that $L \cap \partial D = a_1 \cup a_2$ gives a 0-tangle and $F \cap \partial B$ is isotopic to a $1/0$-tangle. 
Consider the knot $K$ obtained from $L$ 
by replacing $a_1 \cup a_2$ with a $1/n$-tangle. 
Suppose that $| n | > 2$ is odd if $\alpha$ connects parallel arcs, and even otherwise. 
Here $\alpha$ is said to \textit{connect parallel arcs} (resp. connect antiparallel arcs) 
if the orientations of $a_1$, $a_2$ points to the same direction (resp. the different direction). 
Then, for any non-meridional slope $r$, 
the surgered manifold $K(r)$ contains an essential lamination. 

We need to check whether it is applicable to our setting. 
Let $k_0$ denote the knot component of $L_0$.
From $L_0$, a 2-component link $L_1= k_0 \cup A$ is obtained 
by twisting along augmented circles other than $A$.
Here we further require the twisting above satisfies Condition \ref{cond.alt}.
For simplicity, we only consider the case where the surgery slopes for $A$ are $-1/m$ for some $m>0$.

Let $L$ be the link obtained from $L_1$ by 
replacing the tangle corresponding to $A$ with the 0-tangle (without augmented circles). 
Here we regard the tangle corresponding to $A$ as the $1/0$- or $-1/1$-tangle by ignoring $A$. 
See Figure \ref{fig.withalpha}. 

\begin{figure}[htb]
\includegraphics[bb = 0 0 461 72 , scale = 0.75]{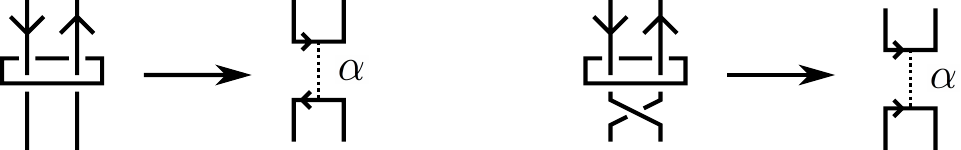}
\caption{Replacement to 0-tangles.}\label{fig.withalpha}
\end{figure}

Note that the diagram $D$ of $L$ so obtained must be alternating, and has an orientation induced from that of $k_0$. 
Since any the 9 graphs in Figure \ref{fig.graphs} has no cut vertex, 
$D$ has to be non-split, and so, $L$ is a non-split oriented link by \cite{Menasco1984}. 
Let $\alpha$ be the arc connecting the 2 strands of the 0-tangle replaced from the tangle corresponding to $A$. 
See Figure \ref{fig.withalpha}. 
Consider the checker-board surface $F$ for the diagram $D$ containing $\alpha$. 
Then, by \cite[Lemma 2.1]{DelmanRoberts1999}, $F$ is a $\pi_1$-injective spanning surface for $L$ 
if $D$ is a reduced alternating diagram of~$L$. 

Assume for the sake of a contradiction that $D$ is not reduced. 
Then the vertex substituted by the tangle corresponding to $A$ 
and the vertex corresponding to the reducible crossing in the graph gives a pair of cut vertices. That is, removing the pair of vertices, the graph must become disconnected. 
Among the 9 graphs in Figure \ref{fig.graphs}, 
only for the graphs $G_7$ and $G_8^1, \dots, G_8^6$ might this arise, and
then only if the pair of vertices is the pair adjacent to a bigon. 
However, such a reducible vertex can appear only if 
the knot obtained by performing twistings on the augmented circles from $L_0$ 
has a reduced alternating diagram 
with twist number less than the number of vertices of the plane graph 
used to construct $L_0$. 

Thus, otherwise, it follows that $D$ is a reduced alternating diagram of $L$, and 
$F$ is a $\pi_1$-injective spanning surface for $L$. 

Now we note that performing Dehn surgery on $A$ along the slope $-1/m$ 
is equivalent to performing replacement $a_1 \cup a_2$ 
with $ 1/(2m - 1)$-tangle (resp. $ 1 / (2m)$-tangle) 
if $\alpha$ connects parallel arcs (resp. antiparallel arcs). 
Thus if an alternating knot $K$ is obtained from $L_1$ 
by Dehn surgeries on $A$ along the slope $- 1/m$, 
then $K$ is obtained from $L_1$ by replacing $a_1 \cup a_2$ with the $1/n$-tangle 
with $n=2m-1$ (resp. $n = 2m$) 
if $\alpha$ connects parallel arcs (resp. antiparallel arcs). 

Consequently, we can apply \cite[Corollary 6.9]{Wu2012} to the setting above to obtain that, 
under the assumption of the lemma, 
for an alternating knot $K$ obtained from $L_1$ 
by Dehn surgeries on $A$ along the the slope $-1/m$, 
the surgered manifold $K(r)$ contains an essential lamination for any non-trivial slope $r$. 

Next we show that the laminations in the surgered manifold $K(r)$ so obtained are all \textit{genuine}, 
i.e., it is carried by an essential branched surface 
with at least one complementary component which is not an $I$-bundle. 
To see this, as claimed in the proof of \cite[Corollary 6.9]{Wu2012}, 
we note that one of the complementary component of the essential branched surface in $K(r)$ 
is the same as the exterior of $L$ cut along $F$. 
Then, by a work of Adams \cite[Theorem 1.9]{Adams2007}, 
or its generalization \cite[Theorem 1.6]{FKP}, 
we see that $F$ is not a fiber surface, in particular, 
the exterior of $L$ cut along $F$ is not an $I$-bundle. 
Thus the laminations in the surgered manifold $K(r)$ so obtained are all genuine. 

This implies that the surgered manifold $K(r)$ is not a small Seifert fibered space 
due to the result by Brittenham \cite{Brittenham1993}. 
Suppose that some Dehn surgery on the alternating knot which we consider in the lemma yields a non-hyperbolic manifold.
Then, as explained in the proof of Corollary \ref{Cor}, it is already known that the manifold must be irreducible, since any hyperbolic alternating knot has no reducible surgeries.
If the manifold is toroidal, then, also as explained in the proof of Corollary \ref{Cor}, it is known that the knot must admit a reduced alternating diagram with twist number at most three, less than the number of vertices of the plane graph used to construct $L_0$.
Thus it suffice to consider the case that the surgered manifold is small Seifert fibered.
Thus the proof of the lemma is completed. 
\end{proof}

By Lemma \ref{cor.Wu}, we can perform twisting along some augmented circles beforehand with slope $1/1$ or $-1/1$.
This reduces the number of components of the links. 

Furthermore, by this twisting, the twist number may decrease, see Figure \ref{fig.dtwist}.
In this case, we do not need to investigate the link. 
Thus we can reduce the number of links to investigate as well. Thus it gives an additional condition. 

\begin{figure}[htb]
\includegraphics[bb =  0 0 495 119 , scale = 0.65]{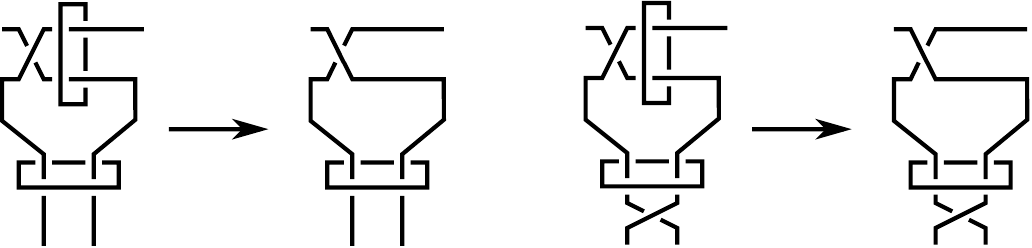}
\caption{The number of twists will decrease.}
\label{fig.dtwist}
\end{figure}

\subsection{Code}\label{subsec.codes}

The whole procedure, namely generating links and applying symmetry and Lemma \ref{cor.Wu}, is implemented as a set of files \texttt{twistLink.py}'s. 

Here, in the case of 7 twists, we summarize the code \texttt{twistLink7.py} in Algorithm \ref{alg.plink}.
We can check if a given sequence satisfies Condition \ref{cond.knot} by carefully tracing the knot components and checking if it is connected.
The code for 6 or 8 twists case works similarly.
We note that for each sequence, 
we first make a .lnk file which is a SnapPy's format for links drawn on plink and then, read that file on SnapPy to get a
triangulation of the link complement.
We perform Dehn surgery if we can apply Lemma \ref{cor.Wu} and as a result we save a .tri file that is a SnapPea's format
for a triangulation.
For .lnk and .tri files, see the documentation of SnapPy \cite{CDW}.
\begin{algorithm}
\caption{Enumerating augmented links}
\label{alg.plink}
\begin{algorithmic}
\REQUIRE Graph $G_7$.
\ENSURE Triangulation files of 1271 link complements.
	\STATE Prepare a list, named checked.
	\FOR{a sequence $\{a_i\}_{i=0}^6$ of $\{0,1,2,3\}$ of length 7}
		\STATE Check if $\{a_i\}_{i=0}^6$ is in is\_checked file. If it is, we go to the next sequence.
		\IF{$\{a_i\}_{i=0}^6$ satisfies Conditions \ref{cond.knot} and \ref{cond.paraaug}}
			\STATE Add symmetric sequences 
			\begin{itemize}
				\item	$a_6(a_2+2)(a_1+2)a_3(a_5+2)(a_4+2)a_0$,
				\item $a_6(a_4 + 2)(a_5 + 2)a_3(a_1 + 2)(a_2 + 2)a_0$.
			\end{itemize}
			 in checked list.
			\STATE Draw edges of $G_6$ depicted in Figure \ref{fig.6-twist}.
			\STATE Determine the sign of the crossings in the rectangles if any and 
				the number $n$ of slopes that will be surgered with 
				negative slopes. (see explanation in \S \ref{subsec.conditions}.)
			\STATE Fill rectangles according to $\{a_i\}_{i=0}^6$ (Figure \ref{fig.fill_tangle}).
			\COMMENT{So far, we are dealing with .lnk file.}
			\FOR{$0\leq j \leq 6$}
				\IF{Two edges passing through the augmented circle in $j$-th rectangle are anti-parallel 
				(Figure \ref{fig.antipara})}
					\STATE Perform a surgery on the augmented circle in $j$-th rectangle along the slope $1/1$ or $-1/1$.
					\COMMENT{Applying Lemma \ref{cor.Wu}. Here we use .tri file. We choose the sign so that Condition \ref{cond.alt} is satisfied.}
					\IF{After the surgery, the number of twists decrease (Figure \ref{fig.dtwist})}
						\STATE Skip this sequence and go to next sequence.
					\ENDIF 
				\ENDIF
			\ENDFOR
		\STATE Save the resulting triangulation file.
		\ENDIF
	\ENDFOR
\end{algorithmic}
\end{algorithm}

In Table \ref{Table.NumberofLinks}, we summarized the number of links that we obtained by running the code.

\begin{table}[htb]
  \begin{center}
    \caption{Number of links}
    \label{Table.NumberofLinks}
    \begin{tabular}{|l|c|c||r|} \hline
     Graph &  number of links \\ \hline \hline
      $G_6$ &   185\\\hline	
      $G_7$&   1271\\\hline	
      $G_8^s$&   2171\\ \hline	
      $G_8^1$&   11299\\\hline	
      $G_8^2$&   5503\\\hline	
      $G_8^3$&   2645\\\hline	
      $G_8^4$&   1651\\\hline	
      $G_8^5$&   3004\\\hline	
      $G_8^6$&   2675\\\hline	
      Total & 30404\\\hline	
    \end{tabular}
  \end{center}
\end{table}

Consequently we have proved Lemma \ref{lem43}.


\section{Procedures for Computer-aided calculations}\label{sec.codes}

In this section, we give detailed explanations on our code \texttt{fef.py}, 
and give a proof of Lemma \ref{lem44} established by computer-aided calculations using the code.

\subsection{A discussion of the code}
We here explain some features of our code \texttt{fef.py}. 

This is based on the code developed in \cite{MartelliPetronioRoukema}. 
In \cite{MartelliPetronioRoukema}, using the code, 
the authors gave a complete classification of 
exceptional fillings of the minimally twisted five-chain link complement, 
a well-known hyperbolic link of 5 components in $S^3$. 
However we had to modify the code, since 
the code in \cite{MartelliPetronioRoukema} essentially depends upon 
the code given in \cite{HMoser}. 
To obtain mathematically rigorous computations that cooperate well with the program
we improved their code using verified numerical analysis based on interval arithmetic. 
Some fundamentals about such method will be given in Appendix \ref{ape.A}. 
Note that the key step to show the links have no exceptional surgeries heavily depends upon the techniques developed in \cite{HIKMOT} by the team containing the authors of this paper.

The main algorithm of \cite{HIKMOT} returns not only a certificate of hyperbolicity for a given
 triangulated manifold $M$, but also closed sets in $\mathbb{C}$ that contain the exact tetrahedral 
 shapes. The methods of that paper use interval arithmetic to establish this claim. While technically 
 the shapes are in a (real valued) interval cross (real valued) interval, i.e. a rectangular box, we will 
 slightly abuse notation and say that each tetrahedral parameter is determined up to an interval. As 
 noted in \cite{HIKMOT}, one of the advantages of interval arithmetic is that it naturally 
 extends to the computations of other invariants and geometric data of the manifold $M$. 
 
One such 
 piece of geometric data is \emph{parabolic length},  i.e. given a horoball packing of $M$ that is 
 maximal in the sense that each horoball in the packing is tangent to at least one other horoball, we 
 can define the length of a parabolic element $p$ fixing a horosphere $S$ setwise as the translation displacement 
 of $p$ in $S$. For a one cusped manifold, this length is canonically defined, however if $M$ has
  more than one cusp, this quantity will depend on our choice of horoball packing. The 6-Theorem, proved
   independently by Agol \cite{agol_dehn} and Lackenby \cite{lack},  provides a key application of 
   parabolic length to our problem namely if $M$ is filled along a sufficiently long slope $r$, then the filled manifold $M(r)$ is hyperbolic.
   If $M$ is filled along multiple cusps by slopes $s_1,...,s_n$, we pay specific attention to 
   \cite[Theorem 3.1]{lack}, which states that
   filling is hyperbolic provided in some horoball packing each slope, $s_i$, we fill along has length strictly bigger than $6$. 
   In the actual computations, we actually enumerate slopes of length less than $6.0001$.
Note that even if we use interval arithmetic, we need to compare floating point numbers, which are the ends of intervals, to prove inequality.
Here we use 6.0001 instead of  6 because floating point numbers are not designed to be used for equality.
  
The code \texttt{fef.py} enumerates all sets of such slopes for a manifold $M$. 
As mentioned above, parts of our code and \texttt{fef.py} in particular are very directly adapted from the code explained \cite[$\S$2.1]{MartelliPetronioRoukema} (compare to their find\_exceptional\_fillings.py). 
Pseudo-code for \texttt{fef.py} is provided as Algorithm \ref{fef.py}.

\begin{algorithm}
\caption{The algorithm for fef.py}
\label{fef.py}
\begin{algorithmic}
\REQUIRE A triangulation $T$ of a manifold $N$. 
\ENSURE A verification that all non-trivial Dehn surgeries of a manifold fitting the conditions of \S\ref{subsec.conditions}
 are hyperbolic.   
	\STATE Try to canonize $T$. 
	\IF {T can be canonized and hikmot verified the hyperbolicity of canonized triangulation.}
		\STATE Use the canonized triangulation.
	\ELSIF {Find a triangulation whose hyperbolicity is checked by hikmot.}
		\STATE Use the found triangulation.
	\ELSE
		\STATE If we cannot find any triangulation that hikmot verifies hyperbolicity, we give up.
		(This didn't happen in our computation for alternating knots)
	\ENDIF
		\STATE Compute lower bounds for the cusped areas of $N$ using the (already) verified tetrahedral shapes for $T$. For each cusp,
		also compute the cusp shape as an parallelogram determined by a quotient of the complex plane by $1$ and $x+yi$. Finally, compute a lower bound for the diameter of the horoball for that cusp and enforce with this bound that the intersection of the boundary of a horoball (not centered at $\infty$) and an ideal tetrahedron having a vertex at $\infty$ intersect in a triangle.
	\IF{Failed on some procedure above.}
		\STATE Use $\frac{3\sqrt{3}}{8}$ as a lower bound for cusp area 
		(horoball of this size always exists by a standard fact of hyperbolic 3-manifolds, see e.g. Proposition 2 in cusp\_neighborhoods.cc of Weeks' SnapPea kernel code, available at the webpage of SnapPy \cite{CDW}). 
		For these cusps, the cusp shape is determined by $1$ and
		$x+yi$ with $x=0$ and $y=\frac{3\sqrt{3}}{8}$.
	\ENDIF
	\STATE 
	The length of a slope $\frac{p}{q}$ is $\sqrt{\frac{A}{y}((p+xq)^2+(yq)^2}$, where $A$ is the area of corresponding horosphere.
	List all slopes of length less than 6.0001 in these cusps. For slopes on each cusp less than length 6.0001,
	perform surgery with that slope if it meets Condition \ref{cond.alt} of \S \ref{sect:conditions}.
	\IF{All cusps have been surgered along.}
		\STATE Verify that the surgered manifold is hyperbolic. 
	\ELSE
		\STATE Verify this intermediately surgered manifold is hyperbolic and repeat the procedure above to find all slopes of length less than 6.0001 in the cusps of this partially surgered manifold and (recursively) verify the hyperbolicity of these surgeries. 
	\ENDIF
\end{algorithmic}
\end{algorithm}

As noted above, \texttt{fef.py} is based upon the file \emph{find\_exceptional\_fillings.py} 
used in the first version of \cite{MartelliPetronioRoukema}.
In the latest version of \emph{find\_exceptional\_\linebreak fillings.py}, they integrated our code.
Although the old version of \emph{find\_\linebreak exceptional\_fillings.py}  is not available anymore,
we give some explanation to clarify our contribution.
For the purposes of the following discussion \texttt{fef.py} will be used to denote our file 
and \emph{find\_exceptional\_fillings.py} will be used to denote the code used for the first version of \cite{MartelliPetronioRoukema}.
One of the key differences between the two files is that \texttt{fef.py} is written to employ interval arithmetic.
However, while the significance of this change is seemingly only visible in a few places such as the declarations of variables, 
it underpins the error control we employ to make the computation rigorous.
The two methods also differ in selecting a horoball packing to compute parabolic length. 
The manifolds we are interested in have a distinguished cusp, namely that corresponding to the knot component, 
whereas the manifolds in \cite{MartelliPetronioRoukema} do not. 
Furthermore, \emph{find\_exceptional\_fillings.py} and the arguments surrounding its implementation cut down 
the number of cases by using the symmetries of the minimally twisted five chain link, 
and so there is a preference toward keeping the horoball packing as symmetric as possible.

To better reduce the number of cases we must consider, we have found it (experimentally) advantageous to
choose a horoball packing where the equivalence class of horoballs corresponding to the cusp of the knot 
component has as much volume as possible so that the slopes in that cusp are as long as possible.
 Consequently, \texttt{fef.py} inflates this horoball past the point of all horoballs being equal volume and 
 reduces the volumes of the other horoballs accordingly. This also produces a horoball packing that is invariant
 under the symmetries described in \S \ref{subsect:symmetry}, and so is very must in the same spirit as
 symmetry reductions employed in \cite{MartelliPetronioRoukema}.
Furthermore, this optimization appears to be crucial since we are dealing with 30404 links and  
we want to reduce the number of slopes to compute as far as possible.

\subsection{Computation environments and verifying computations}

As we have seen in the previous subsection, the code \texttt{fef.py} computes recursively with respect to the number of cusps.
Usually, for each augmented circle of our links generated as explained in \S \ref{sec.symmetry}, 
there are 2 or 3 surgery slopes of parabolic length less than 6.
In the worst case, there are about $18,000$ manifolds to investigate for a single link.
In this case, it takes about 51 hours on 
a single CPU of TSUBAME (The computational ability of a single CPU of TSUBAME is 
comparable to that of a standard personal computer).
Since there are 30404 links, we need high-spec machine.
The second author was able to access "TSUBAME", the super-computer of Tokyo Institute of Technology.
See the website \cite{TSUBAME.web} of TSUBAME for a basic information, 
\cite{TSUBAME1} for a brief survey, and \cite{TSUBAME2} for a detailed exposition.
Roughly speaking, on TSUBAME, one can use many machines at the same time.
Although generally, to use parallel computation effectively we need some work, in our case, the situation
itself is totally parallel, that is, we need to investigate each link independently.
Thus we can use TSUBAME effectively.
In practice, we ``rented" 64 machines from TSUBAME, and then it took a week to prove Theorem \ref{MainThm}.

\begin{proof}[Proof of Lemma \ref{lem44}]
Let $L$ be one of the 30404 augmented links in $S^3$ obtained in Lemma \ref{lem43}, and 
$K$ an alternating knot obtained by a Dehn surgery on $L$ 
such that the surgery corresponds to twisting along the unknotted components of $L$. 
By construction, $K$ has a reduced alternating diagram with twist number at least $6$, and so it is not a torus knot of type $(2,p)$. 
This implies that $K$ is hyperbolic by \cite{Menasco1984}. 

We show that $K$ admits no non-trivial exceptional surgeries by running our code \texttt{fef.py} on TSUBAME 
for the triangulation files of the 30404 augmented links obtained by the files \texttt{twistLink.py}'s. 

On TSUBAME, we need to know a specified command to run it. The command we used is

\medskip
\noindent \texttt{t2sub -q V -J 0-11299 -l walltime=5:00:00  -W group\_list=t2gxxx\\ -l  select=1 ./8twist1.sh.}

\medskip
\noindent Note that the command should be in one line.
We here explain this command.
First, ``t2sub" is the basic command to run TSUBAME and we used several options;
\begin{itemize}
\item ``-q V" specifying a queue name to submit a job (always necessary).
\item `` -I walltime = 5:00:00" meaning that if the computation time exceeded 5 hours, then we quit the computation.
\item ``-W group\_list=t2gxxxxxxx" specifying the name of the user.
\item ``-I select=$n$" meaning that for a single computation (i.e. a single link in our case), we use ``$n$ machines", and
\item ./8twist1.sh is the execution file. The contents are as follows;
	\begin{itemize}
		\item[] \#!/bin/sh
		\item[] cd \$\{PBS\_O\_WORKDIR\}
		\item[] python fef.py 8twist1tri/dataname\$PBS\_ARRAY\_INDEX.tri
	\end{itemize}
	Here \${PBS\_O\_WORKDIR} is the current directly and,
\item ``-J 0-11299" means that \$PBS\_ARRAY\_INDEX ranges from 0 to 11299.
\end{itemize}
We remark here although \texttt{twistLink.py}'s generate the triangulation files named like 0\_021213.tri, 
(here the first 0 is the number of augmented circles that will be filled with $1/p$ with negative $p$)
we renamed all the triangulation files to ``dataname$n$.tri" so that it will be suitable for ``-J" option above, called array job.
The information about the number of augmented circles 
that will be filled with negative slopes are stored in the contents of the files.

TSUBAME returns an outputs file and an error file for each triangulation file.
We here include examples,
\begin{itemize}
\item output file\\
8twist1tri/dataname1015.tri\\
./4\_20320113.lnk\_filled(0,0)(0,0)(0,0)(0,0)\\
Manifold volume:\\
31.5182982095\\
Cusp shapes:\\
 Cusp 0 : complete torus cusp of shape (5.02637698521+9.96657568374j),\\
 Cusp 1 : complete torus cusp of shape (0.594148461042+0.999088628226j),\\
 Cusp 2 : complete torus cusp of shape (0.333034596342+1.15348583079j),\\
 Cusp 3 : complete torus cusp of shape (0.333034596342+1.15348583079j)\\
With 4 fillings:\\
Total: 0\\
Candidate hyperbolic fillings:\\
With 4 fillings:\\
$[]$\\
Total: 0\\
8twist1tri/dataname1015.tri done\\
Computer time needed: 0:00:16.603389\\
Number of manifolds hikmot ensured the hyperbolicity 19\\
\item error file\\
======================================\\
Your accounting ID\\
 group id            : t2xxxxxxx\\
------------------------------------------------\\
 Job informations\\
 job id              : 60373[1015].t2zpbs-vm1\\
 queue               : V\\
 num of used node(s) : 1\\
 used node(s) list   :\\
  t2a001137-vm1\\
 used cpu(s)         : 1\\
 walltime            : 00:00:16 (16 sec)\\
 used memory         : 26080kb\\
 job exit status     : 0\\
------------------------------------------------\\
Accounting factors\\
 x 1.0 by queue\\
 x 1.0 by job priority\\
 x 1.0 by job walltime extension\\
 = 1.0 is the total accounting factor\\
------------------------------------------------\\
Expense informations\\
 maximum CPU units with factors : 64\\
 used CPU units : 1\\
======================================\\
\end{itemize}
This proves the link obtained from $G^1_8$ and the sequence $20320113$ does not have any exceptional surgeries satisfying Condition \ref{cond.alt}.
By running \texttt{fef.py} on TSUBAME, we have 30404 output files and error files.
\footnote{Unfortunately the version 1.1 of \texttt{fef.py} contained a small bug, which could be immediately fixed in the update version (fef\_ver1.4).
It was not relevant to the main procedure, but for rigorousness, we have performed re-computation, and obtained the same outcomes (in February, 2016).
Please see the readme file of \texttt{fef.py} available on our web site \cite{Webpage}.}
These data are available at \cite{Webpage}. 
The versions of gcc and python, and any other relevant information on the nodes of  TSUBAME that we used are shown in Table~\ref {version_table}. 
Fortunately, for all 30404 manifolds, the outputs show that they have no 
exceptional surgeries satisfying Condition \ref{cond.alt}.
In total, i.e. the sum of the computation time of all nodes, 
computation time was approximately 512 days, and 
the number of manifolds we applied hikmot is 5646646.
Consequently we have completed our proof of the Lemma \ref{lem44}. 
\end{proof}

\begin{table}[htb]
    \begin{center}
    \caption{A description of the computer system used for this computation}
    \label{version_table}
    \resizebox{\textwidth}{!}{
    \begin{tabular}{|l|l|}\hline
SnapPy & snappy-1.3.12-py2.6-linux-x86\_64.egg\\
HIKMOT & HIKMOT\_ver0.1.0\\
FEF & fef\_ver1.1\\
PYTHON & python2.6\\
\hline
\multicolumn{2}{|c|}{TSUBAME(version 2) :}\\
\hline
OS & SUSE Linux Enterprise Server 11 SP3\\
Job Scheduler & PBS Professional\\
Compilers & 
Intel Compiler 2013.1.046 (default), PGI CDK 14.6, gcc 4.3.4\\
MPI & 
OpenMPI 1.6.5 (default), MVAPICH2 2.0rc-1\\
CUDA & 6.0.1\\
CUDA driver & 331.62\\
\hline
    \end{tabular}}
    \end{center}
  \end{table}

\section*{Acknowledgements}
The authors would like to thank Neil Hoffman, 
Bruno Martelli, Nathan Dunfield, and Marc Culler for helpful discussions. 
They also thank the Support team of TSUBAME for providing technical support, 
and Shin'ichi Oishi, Masahide Kashiwagi, and Akitoshi Takayasu for their help with the 
building the verification tools needed to implement this project. 
They would also like to thank Matthias Goerner for catching a bug in the earlier version of our program. 
Finally they thank the referees for their careful readings and useful comments.


\appendix

\section{Verified Computation}\label{ape.A}
In this section, we recall briefly the notion of {\em interval arithmetic}, that makes it possible to prove rigorously
inequality by using computer.
In \texttt{fef.py}, we use so-called 6-Theorem, that states that if a parabolic length of a slope is greater
than 6, then the surgery along that slope gives a hyperbolic manifold.
Hence to study exceptional surgeries, 
we only need to consider slopes of length less than 6, and here we need to prove inequality.

On usual computation, we use {\em floating point arithmetic}.
The floating point arithmetic is a very practical method to perform approximated computation.
Here we do not go into detail, instead let us note that the set $\mathbb{F}\subset\mathbb{R}$ of real numbers that can be
represented by floating point arithmetic satisfies $|\mathbb{F}|<\infty$.
There are several ways to define $r:\mathbb{R}\rightarrow\mathbb{F}$, which is called a rounding operator.
Here we would have some error that might accumulate by iterating this process.
Interval arithmetic is introduced to deal with this unpleasant error \cite{Sunaga1,Sunaga2,Moore}.
In interval arithmetic, instead of dealing with approximated values, 
we use closed intervals of type $X=[\underline{x},\overline{x}]$.
We denote the set of all intervals by $\mathbb{IR}$.
We will design our computation as follows so that each of our intervals contains the exact value.
We first recall abstract theory of interval arithmetic and later, we will explain so-called machine interval arithmetic.
First, for a given function $f:\mathbb{R}\rightarrow\mathbb{R}$, a function $F:\mathbb{IR}\rightarrow\mathbb{IR}$ is
said to be an interval extension of $f$ if
$$F(X)\supset \{f(x)\mid x\in X\},~\forall X\in \mathbb{IR}.$$
We remark here that in \texttt{fef.py}, the only function we use is the square root, a monotone function.
Hence in practice we only need to consider endpoints of a given interval.
We further define interval extensions of four arithmetic operations as follows;
$$\begin{array}{ll}
	X+Y&=\left[\underline{x}+\underline{y},\overline{x}+\overline{y}\right],\\
	X-Y&=\left[\underline{x}-\overline{y},\overline{x}-\underline{y}\right],\\
	X\cdot Y&=\left[\min\{\underline{x}\cdot\underline{y},
	\overline{x}\cdot\overline{y},
	\underline{x}\cdot\overline{y},
	\overline{x}\cdot\underline{y}\},
	\max\{\underline{x}\cdot\underline{y},
	\overline{x}\cdot\overline{y},
	\underline{x}\cdot\overline{y},
	\overline{x}\cdot\underline{y}\}\right],\mbox{ and }\\
	X/Y&=X\cdot\left[\frac{1}{\overline{y}},\frac{1}{\underline{y}}\right],~(0\not\in Y).
\end{array}$$
To implement the interval arithmetic we need to consider $\mathbb{IF}$, 
the set of closed intervals whose endpoints are elements of $\mathbb{F}$.
We define two rounding operators 
$\lceil\cdot\rceil_{\mathbb{F}}, \lfloor\cdot\rfloor_{\mathbb{F}}:\mathbb{R}\rightarrow\mathbb{F}$ as
$$\lceil x\rceil_{\mathbb{F}} = \min\{y\in\mathbb{F}\mid x\leq y\},
~\lfloor x\rfloor_{\mathbb{F}} = \max\{y\in\mathbb{F}\mid x\geq y\}.$$
Then we can define a rounding operator $\Box:\mathbb{IR}\rightarrow\mathbb{IF}$ as
$\Box([\underline{x}, \overline{x}])\linebreak = \left[\lfloor \underline{x}\rfloor_{\mathbb{F}}, \lceil \overline{x}\rceil_{\mathbb{F}}\right]$.
Then for a given interval extension $F$ of $f$, we define the machine interval extension
$\bar{F}:\mathbb{IR}\rightarrow\mathbb{IF}$ of $F$ by $\bar{F}(X) = \Box(F(X))$.
Similarly for $\circ\in\{+,-,\times,/\}$, the machine interval extension is defined as $X\bar{\circ}Y = \Box(X\circ Y)$.
Thus we can handle round off errors.
If we use machine interval arithmetic, 
it can be readily seen that the exact value is always contained in the interval that we compute.
Therefore we can rigorously prove inequality by comparing suitable endpoints of resulting intervals.
This enables us to prove inequality by computer.

\section{A family of Montesinos knots}\label{ape.B}

We here include an application of our method 
to obtain a complete classification of exceptional surgeries on 
the Montesinos knots $M(-1/2, 2/5, 1/(2q + 1))$ with $q \ge 5$. 
Consequently the Montesinos knots are shown to have no non-trivial exceptional surgeries. 
This gives the last piece for a complete classification of exceptional surgeries 
on hyperbolic arborescent knots in $S^3$. 

Let us apply the code \texttt{fef\_mon.py} (also available at \cite{Webpage}) to the link $L_M$ depicted in Figure \ref{fig.Montesinos} (left). 
The code \texttt{fef\_mon.py} is essentially the same as \texttt{fef.py}, 
but it is suitably tuned for investigating $L_M$. 
It requires two input, namely the name of the manifold and
the number of augmented circles that will be surgered along $1/p$ with $p<0$, while 
\texttt{fef.py} automatically reads such number from the .tri files we generated by \texttt{twistLink.py}'s.

\begin{figure}[htb]
\includegraphics[bb = 0 0 399 94, width=\textwidth]{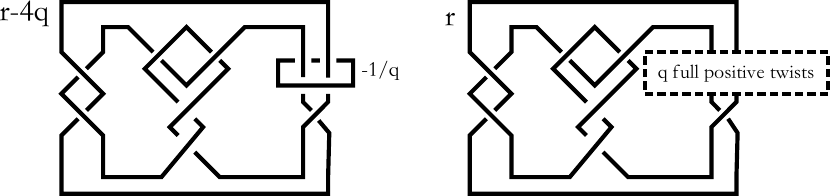}
\caption{(left)$L_M$, (right)$M(-1/2, 2/5, 1/(2q + 1))$.}
\label{fig.Montesinos}
\end{figure}

\begin{remark}
Since the linking number of the two components of $L_M$ is $2$, if we perform Dehn surgery on the augmented circle of $L_M$, 
the surgery slope will be twisted. 
For this reason, for Dehn surgery on $M(-1/2, 2/5, 1/(2q + 1))$ with slope $r$, the corresponding
slope on $L_M$ will be $r-4q$. See \cite{RolfsenBook} for details.
\end{remark}

The augmented circle will be filled by slope $-1/q$ and the other component will be filled by slope $r$.
By our code, we see that $L_M$ does not admits any exceptional surgeries with $q>5$.
The output is available in our web site \cite{Webpage}. 

Although our code can enumerate all exceptional surgeries, for the case of $L_M$, the list that our code returns 
contains many redundant, i.e. hyperbolic surgeries.
Hence we will apply our code for the case of $q=1,2,3,4,$ and $5$ separately by directly drawing diagrams.
We summarize the result in Table~\ref{Table.Mont}.
\begin{table}[htb]
  \begin{center}
    \caption{Exceptional fillings}
    \label{Table.Mont}
    \begin{tabular}{|l|c|c||r|} \hline
     $q$ &  $r$, candidate exceptional & $r$, candidate hyperbolic \\ \hline \hline
      1 &   3,4,5,6 & 7\\
      2 &   7,8,9 & 6, 10\\
      3 &   11,12 & 10,13\\ 
      4 &   15 & 14,16\\
      5 &   $\emptyset$ & 18,19,20\\
      \hline	
    \end{tabular}
  \end{center}
\end{table}

The candidate exceptional fillings in Table \ref{Table.Mont} are all known to be exceptional. 
In fact, $K(r)$ is toroidal if $(q,r) = (1,6), (2,9), (3,12), (4,15)$. 
See \cite[Theorem 1.1]{Wu2011a}. 
Otherwise $K(r)$ is small Seifert fibered. 
See \cite[Theorem~3.2]{WuPreprint}, and also see \cite{Meier}. 
(Recall that Montesinos knots have no toroidal Seifert fibered surgeries \cite{IchiharaJong2013}.)

Hence to complete the classification, it suffices to prove that all candidate hyperbolic fillings in Table \ref{Table.Mont}
actually give hyperbolic manifold.
We will use the following algorithm that we used in \cite{HIKMOT} to verify the hyperbolicity of Hodgson-Weeks Closed Census \cite{HW}.
\begin{algorithm}                      
\caption{Find positive solutions by drilling out}
\label{alg.drillout}
\begin{algorithmic}
\REQUIRE $M$ is a closed manifold with a surgery description.
\ENSURE $M$ has a good triangulation.
	\WHILE{We could find a short closed geodesic $\gamma\subset M$}
			\STATE Drill out $\gamma$ to get $M\setminus\gamma$,
			\STATE Take filled\_triangulation $N$ of $M\setminus\gamma$,
			\STATE Fill the cusp of $N$ by the slope $(1,0)$.
			\STATE (By the above procedure, we forget original surgery description and get new surgery description.)
			\IF{$N$ has positively oriented solution.}
				\IF {hikmot \cite{hikmot} verifies the hyperbolicity of $N$}
					\STATE return [True, $N$]	
				\ENDIF
			\ENDIF
	\ENDWHILE
\STATE return False.
\end{algorithmic}
\end{algorithm}
The main idea of Algorithm \ref{alg.drillout} is due to Craig Hodgson.
The code \texttt{makepositive\_drill.py} available at \cite{Webpage} 
implements the algorithm.
Then by using the code, we can verify the hyperbolicity of all resulting manifolds of candidate hyperbolic fillings.
Thus we complete the classification of exceptional surgeries along Montesinos knots.

This gives the last piece for a complete classification of exceptional surgeries 
on hyperbolic arborescent knots in $S^3$ as follows. 
Any hyperbolic arborescent knot of type III has no exceptional surgeries as shown by Wu \cite[Theorem 3.6]{Wu1996}. 
A complete classification of exceptional surgeries on hyperbolic arborescent knot of type II 
is obtained also by Wu \cite{Wu2011b}. 
There are just 3 knots among them admitting exceptional surgeries, which are all toroidal. 
Hyperbolic type I arborescent knots are two-bridge knots and Montesinos knots of length three. 
For two-bridge knots, 
a complete classification of exceptional surgeries is obtained by Brittenham and Wu \cite{BrittenhamWu2001}. 
The remaining case, for Montesinos knots of length three, other than $M(-1/2, 2/5, 1/(2q + 1))$ with $q \ge 5$, 
a complete classification of exceptional surgeries is recently established by Meier \cite{Meier}. 
Now we have shown that $M(-1/2, 2/5, 1/(2q + 1))$ with $q \ge 5$ have no exceptional surgeries. 

\bigskip

We here include a summary. 
See \cite{IchiharaJong2009}, \cite{IchiharaJong2013}, \cite{Meier}, \cite{Wu1996}, \cite{Wu2011a}, \cite{Wu2011b}, \cite{Wu2012}, \cite{WuPreprint} for details. 

Let $K$ be a hyperbolic arborescent knot in $S^3$. 
Suppose that the manifold $K(r)$ obtained by Dehn surgery on $K$ along a non-trivial slope $r$ 
is non-hyperbolic for some rational number $r$. 
Then $r$ must be an integer except for $r=37/2$ for $P(-2, 3, 7)$. 
The manifold $K(r)$ is always irreducible, and has infinite fundamental group except for 
$r= 17,18,19$ for $P(-2, 3, 7)$ and $r= 22,23$ for $P(-2, 3, 9)$. 
Furthermore the following hold. 
If $K(r)$ is toroidal, then $K(r)$ is not a Seifert fibered, and $K$ is either 
\begin{itemize}
\item
a two bridge knot $K_{[b_1,b_2]}$ with $|b_1|,|b_2| > 2$, 
and $r=0$ if both $b_1 , b_2$ are even, 
$r = 2 b_2$ if $b_1$ is odd and $b_2$ is even, 
\item
a twist knot $K_{[2n,\pm 2]}$ with $|n| > 1$ and $r= 0, \mp 4$, 
\item
one of the Montesinos knots of length 3 with the slope described in Table \ref{Table.Tor}. 
\begin{table}[htb]
  \begin{center}
    \caption{Toroidal surgeries}
    \label{Table.Tor}
    \begin{tabular}{|l|c|}
    \hline
\multicolumn{1}{|c|}{$K$} &  $r$ \\ \hline \hline
$P(q_1,q_2,q_3)$, $q_i$ odd and $| q_i |>1$ & $0$ \\
$P(q_1, q_2, q_3)$, $q_1$ even, $q_2$, $q_3$ odd and $|q_i| > 1$ & $2(q_2 +q_3)$\\
$P(-2, 3, 7)$ & $37/2$\\
$P( - 3 , 3 , 7 )$ & $1$\\
$M( - 1 / 2 , 1 / 3 , 1 / ( 3 + 1 / n ) )$, $n$ even and $n \ne 0$ & $2 - 2 n$\\
$M(-1/2, 1/3, 1/(5+1/n))$, $n$ even and $n\ne 0$ & $1-2n$\\
$M(-1/2, 1/3, 1/(6 + 1/n))$, $n \ne 0, -1$ odd (resp. even) & $16$ (resp. $0$)\\
$M(-1/2, 1/5, 1/(3+1/n))$, $n$ even and $n \ne 0$ & $5-2n$\\
$M( - 1 / 2 , 2 / 5 , 1 / 7 ) $ & $12$\\
$M( - 1 / 2 , 2 / 5 , 1 / 9 )$ & $ 15$ \\
$M(-1/3, -1/(3+1/n), 2/3)$, $n \ne 0,-1$ odd (resp. even) & $-12$ (resp. $4$)\\
$M(-2/3, 1/3, 1/4)$ & $13$\\
$M( - 1 / ( 2 + 1 / n ) , 1 / 3 , 1 / 3 )$, $n$ odd and $n \ne - 1$& $2 n$\\
      \hline	
    \end{tabular}
  \end{center}
\end{table}

\item
$K_1$ with $r=3$, $K_2$ with $r=0$ or $K_3$ with $r= -3$. 
Here $K_1, K_2,K_3$ are defined as follows. 
Let $T(r_1,r_2)$ be the Montesinos tangle obtained as the sum of rational tangles corresponding to $r_1$ and $r_2$. 
Denote by $T(r_1,r_2; n)$ the tangle obtained from $T(r_1,r_2)$ by twisting the two lower endpoints of the strands by $n$ left hand half twists. 
Let $\eta : \mathbb{R}^2 \to \mathbb{R}^2$ be the map which is a $\pi/2$ counter-clockwise rotation about the origin followed by a reflection along the $y$-axis. 
Define three knots $K_1, K_2, K_3$ obtained as 
$(S^3, K_1) = T (1/3, -1/2; 4) \cup_\eta T (1/3, -1/2; 4)$,
$(S^3, K_2) = T (1/3, -1/2; 4)  \cup_\eta T (-1/3, 1/2; -4)$, and 
$(S^3, K_3) = T (-1/3, 1/2; -4)  \cup_\eta  T (-1/3, 1/2; -4)$.
\end{itemize}

If $K(r)$ is small Seifert fibered, then $K$ is either 
\begin{itemize}
\item
the figure-eight knot and $r= \pm 1 , \pm 2 , \pm 3$, 
\item
a twist knot $K_{[2n,\pm 2]}$ with $|n| > 1$ and $r= \mp 1, \mp 2, \mp 3$, 
\item
one of the Montesinos knots of length 3 with the slope described in Table \ref{Table.SF}. 
\end{itemize}

\begin{table}[htb]
  \begin{center}
    \caption{Seifert fibered surgeries}
    \label{Table.SF}
    \begin{tabular}{|l|c|}
    \hline
\multicolumn{1}{|c|}{$K$} &  $r$ \\ \hline \hline
$P(-2,\, 3,\, 2n+1)$, $ n \ne 0,1,2$ & $4n+6$ , $4n+7$ \\
$P(-2,\, 3,\, 7)$ & $17$  \\
$P(-3,\, 3,\, 3)$ & $1$ \\
$P(-3,\, 3,\, 4)$ &$1$ \\
$P(-3,\, 3,\, 5)$ & $1$  \\
$P(-3,\, 3,\, 6)$ & $1$ \\ 
$M(-1/2,\, 1/3,\, 2/5)$ & $3$ , $4$ , $5$\\ 
$M(-1/2,\, 1/3,\, 2/7)$ & $-1$ , $0$ , $1$\\
$M(-1/2,\, 1/3,\, 2/9)$ & $2$ , $3$ , $4$ \\
$M(-1/2,\, 1/3,\, 2/11)$ & $-1$ , $-2$  \\
$M(-1/2,\, 1/5,\, 2/5)$ & $7$ , $8$\\
$M(-1/2,\, 1/7,\, 2/5)$ & $11$ \\
$M(-2/3,\, 1/3,\, 2/5)$ & $-5$ \\
      \hline	
    \end{tabular}
  \end{center}
\end{table}

%

\end{document}